\documentclass[dvipdfmx]{article}
\usepackage{amsmath,amsthm,amssymb}
\usepackage{mathrsfs,bm}
\usepackage{enumitem, appendix}

\usepackage[dvipdfmx]{hyperref}
\hypersetup{
    colorlinks=true,
    linkcolor=blue,
}

\newtheorem{theo}{Theorem}[section]
\newtheorem{prop}[theo]{Proposition}
\newtheorem{coro}[theo]{Corollary}
\newtheorem{lemm}[theo]{Lemma}

\theoremstyle{definition}
\newtheorem{defi}[theo]{Definition}
\newtheorem*{note*}{Note}
\newtheorem*{claim*}{Claim}
\newtheorem*{exam*}{Example}
\newtheorem*{rema*}{Remark}
\newtheorem{assu}[theo]{Assumption}
\newtheorem{notation}{Notation}

\author{Shota FUKUSHIMA\thanks{Graduate School of Mathematical Sciences, the University of Tokyo, 3-8-1 Komaba, Meguro-ku, Tokyo, 153-8914, Japan. 
Email: fukusima@ms.u-tokyo.ac.jp} 
\thanks{The author is supported by Leading Graduate Course for Frontiers of Mathematical Sciences and Physics (FMSP), at Graduate School of Mathematical Science, the University of Tokyo. }}
\title{Semiclassical pseudodifferential operators and resolvent parametrices on manifolds with ends}

\newcommand{\transp}[1]{{}^t\!{#1}}

\newcommand{\vol}{\mathrm{vol}}
\newcommand{\jbracket}[1]{\left\langle {#1} \right\rangle}

\newcommand{\rmop}[1]{\mathop{\mathrm{#1}}}

\newcommand{\diff}{\mathrm{d}}
\newcommand{\diffrac}[2][]{\frac{\diff #1}{\diff #2}}

\newcommand{\nni}{\mathbb{Z}_{\geq 0}}
\newcommand{\Op}{\mathrm{Op}}

\newcommand{\locsymb}[2]{{\widetilde S}^{#1}_{#2}}

\newcounter{step}
\newcommand{\fstep}[1]{
    \par\vspace{2mm}
    \noindent\textbf{#1.}
}

\newcounter{enumpartcounter}

\numberwithin{equation}{section}

\begin{document}
\maketitle

\begin{abstract}
    We construct a parametrix of a resolvent of elliptic differential operators acting on half-densities on manifolds with ends. The construction is carried out by introducing suitable pseudodifferential operators compatible with the end structure. Our class of pseudodifferential operators and symbols is independent of the choice of Riemannian metric on the manifold and applicable to both of asymptotically conical and hyperbolic manifolds. As an application, we prove the essential self-adjointness of elliptic symmetric differential operators on manifolds. 
    \end{abstract}

    \section{Introduction}\label{def_manifold_with_ends}
Let $M$ be an $n$ dimensional non-compact manifold with the following assumption. 
\begin{assu}\label{assu_manifold_with_end_III}
    There exist an open subset $E$ of $M$, a compact manifold $S$ with dimension $n-1$ and a diffeomorphism $\Psi: E \to \mathbb{R}_+ \times S$ such that $M\setminus \Psi^{-1}((1, \infty)\times S)$ is compact. Here we denote $\mathbb{R}_+:=(0, \infty)$. The set $E$ is called \textit{ends} of $M$. 
\end{assu}

In the following, we fix the mapping $\Psi: E \to \mathbb{R}_+$ in Assumption \ref{assu_manifold_with_end_III}. 
If $E$ is diffeomorphic to $\mathbb{R}_+\times S$, we can consider the ``\textit{polar coordinates}'' $(r, \theta)\in (1, \infty)\times S$ using the local coordinates $\theta=(\theta_1, \ldots, \theta_{n-1})$ on $S$. 
This is defined as the image of local coordinates introduced below. We take an atlas on $ M $ which consists of finite local coordinates $\{\varphi_\iota:  U_\iota \to  V_\iota\}_{\iota\in I}$ in the following way.  

\begin{enumerate}
\item We take a finite atlas $\{\varphi_\iota:  U_\iota\to  V_\iota\}_{\iota\in I_ K }$ on $M\setminus E$ where 
$U_\iota$ and $V_\iota$ are open subsets of $M$ and $\mathbb{R}^n$ respectively. 
\item Recall the diffeomorphism $\Psi: E \to \mathbb{R}_+\times S$. Since $S$ is a compact manifold, we can take a finite atlas $\{\varphi_\iota^\prime:  U_\iota^\prime\to  V_\iota^\prime\}_{\iota\in I_\infty}$ on $S$. Here $U_\iota^\prime$ and $ V_\iota^\prime$ are open subsets of $S$ and $\mathbb{R}^{n-1}$ respectively. We set $U_\iota:=\Psi^{-1}(\mathbb{R}_+\times U_\iota^\prime)$, $V_\iota:=\mathbb{R}_+\times V_\iota^\prime$ and 
\begin{equation}\label{eq_polar_defi_III}
    \varphi_\iota:=(\mathrm{id}\times \varphi_\iota^\prime)\circ \Psi:  U_\iota\longrightarrow  V_\iota.
\end{equation} 
\item Finally we put $I:=I_ K \cup I_\infty$. Here we take $I_K$ and $I_\infty$ such that $I_K\cap I_\infty=\varnothing$. Obviously $I$ is a finite set. 
\end{enumerate}

We introduce a class of functions invariant of the action of normalized differential operators $\partial_r$ and $f(r)^{-1}\partial_\theta$ in a metric $\diff r^2+f(r)^2 h(\theta, \diff \theta)$. 

\begin{assu}\label{assu_f}
    We fix a smooth function $f:  \mathbb{R} \to [1, \infty)$ such that 
    \begin{equation}\label{enum_f}
        \partial_r^j \log f\in L^\infty (\mathbb{R})
    \end{equation} 
    for all $j\geq 1$. 
    \end{assu}

\begin{defi}\label{defi_class_polar}
    Let $V^\prime$ be an open subset of $\mathbb{R}^{n-1}$. A function $a\in C^\infty (\mathbb{R}_+ \times V^\prime)$ belongs to the class $\mathcal{B}_f (\mathbb{R}_+ \times V^\prime)$ if
    \[
(f(r)^{-1}\partial_\theta)^{\alpha^\prime} \partial_r^{\alpha_0} a(r, \theta)\in L^\infty( [R, \infty)\times K)
\]
for all $R>0$, all compact subsets $K\subset V^\prime$ and all multiindices $\alpha=(\alpha_0, \alpha^\prime)\in \nni^n$. 
\end{defi}

Our main result is that we can construct parametrices of elliptic differential operators on $M$. In this paper, we consider (pseudo)differential operators acting on half-densities. The definition of half-densities in Chapter 14 in \cite{Zworski12} for example. Let $P_\hbar : C^\infty(M; \Omega^{1/2})\to C^\infty (M; \Omega^{1/2})$ be a semiclassical differential operator acting on half-densities on $M$. Here $C^\infty (M; \Omega^{1/2})$ is the space of all smooth half-densities on $M$. We assume that $P_\hbar$ is of the form 
\begin{equation}
    \begin{aligned}
        &\varphi_{\iota*} P_\hbar  \varphi_\iota^*(u|\diff r \diff \theta|^{1/2}) \\
        &=\sum_{|\alpha|\leq m}p^\iota_\alpha (\hbar; r, \theta)(f(r)^{-1}\hbar D_\theta)^{\alpha^\prime} (\hbar D_r)^{\alpha_0}u (r, \theta)|\diff r \diff \theta|^{1/2}
    \end{aligned}
 \label{eq_P_local}
\end{equation}
with 
\[
    p^\iota_\alpha (\hbar; r, \theta)=\sum_{j=0}^{N_\alpha}\hbar^j p^\iota_{\alpha, j}(r, \theta)
\]
for some $N_\alpha\in \nni$ on the polar coordinate neighborhood $U_\iota$ ($\iota\in I_\infty$) where $a^\iota_\alpha\in \mathcal{B}_f(V_\iota)$. Here we employed the notation 
\[
    D_r=-i\partial_r, \, D_{\theta_j}=-i\partial_{\theta_j}, \, D_\theta^{\alpha^\prime}=D_{\theta_1}^{\alpha^\prime_1}\cdots D_{\theta_{n-1}}^{\alpha^\prime_{n-1}}, \, \alpha=(\alpha_0, \alpha^\prime), \, |\alpha|=\alpha_0+|\alpha^\prime|. 
\]
We denote the set consisting of such differential operator $P_\hbar$ by $\mathrm{Diff}_{f, \hbar}^m(M; \Omega^{1/2})$. The principal symbol $\sigma(P_\hbar): T^* M \to \mathbb{C}$ of $P_\hbar$ is defined as 
\[
\sigma(P_\hbar)(\tilde \varphi_\iota^{-1}(r, \theta, \rho, \eta))
:=\sum_{|\alpha|\leq m}p^\iota_\alpha (r, \theta)(f(r)^{-1}\eta)^{\alpha^\prime} \rho^{\alpha_0} 
\]
if $\varphi_{\iota*} P_\hbar  \varphi_\iota^*$ is given by \eqref{eq_P_local}. Here $\tilde \varphi_\iota: T^* U_\iota\to  T^*V_\iota$ is the canonical coordinates associated with $ \varphi_\iota:  U_\iota\to  V_\iota$. 
We call $ P_\hbar \in \mathrm{Diff}^m_{f, \hbar} (M; \Omega^{1/2})$ with $\sigma (P_\hbar)(T^*M)\neq \mathbb{C}$ elliptic if, for all $z\in\mathbb{C}$ with $\mathrm{dist}(z, \sigma( P_\hbar )(T^* M ))>0$, there exists a constant $C>0$ such that the inequality
\[
C^{-1}(1+|\rho|+f(r)^{-1}|\eta|)^m\leq |z-\sigma( P_\hbar )(r, \theta, \rho, \eta)| \leq C(1+|\rho|+f(r)^{-1}|\eta|)^m
\]
holds in polar coordinates, and similar condition holds on $M\setminus E$. The precise definition is in Definition \ref{def_elliptic_M}. 

We construct a parametrix of the resolvent $(z-P_\hbar)^{-1}$ of an elliptic differential operator $P_\hbar$. In order to construct parametrices, we introduce a suitable class of pseudodifferential operators in order to investigate such differential operators in polar coordinates. 

\begin{notation}
We introduce shorthand notations 
\[
q=(r, \theta)\in \mathbb{R}^n=\mathbb{R}\times \mathbb{R}^{n-1}
\]
and for the dual variable $(\rho, \eta)\in \mathbb{R}\times\mathbb{R}^{n-1}$ 
\[
p=(\rho, \eta)=\rho\oplus\eta\in \mathbb{R}^n=\mathbb{R}\times\mathbb{R}^{n-1}. 
\]
We also introduce notations
\[
\jbracket{p}:=(1+|p|^2)^{1/2}
\]
and
\[
\jbracket{\rho\oplus\eta}:=(1+\rho^2+|\eta|^2)^{1/2}. 
\]

As in \eqref{eq_P_local}, we always split a multiindex as $\alpha=(\alpha_0, \alpha^\prime)\in \nni\times\nni^{n-1}$. The length $|\alpha|$ is defined as
\[
|\alpha|=\alpha_0+|\alpha^\prime|=\alpha_0+\alpha_1+\cdots +\alpha_{n-1}
\]
if $\alpha^\prime=(\alpha_1, \ldots. \alpha_{n-1})$. 
We denote
\[
\partial_q^\alpha=\partial_r^{\alpha_0}\partial_\theta^{\alpha^\prime}, \, \partial_p^\beta=\partial_\rho^{\beta_0}\partial_\eta^{\beta^\prime}. 
\]
\end{notation}

We define a symbol class on manifolds with ends as follows. 

\begin{defi}\label{defi_symbol_class_manifold}
    Let $m$ be a real number. A smooth function $a\in C^\infty (T^*M)$ belongs to $S^m_f (T^*M)$ if the following two conditions hold. 
    \begin{itemize}
        \item For $\iota\in I_K$ and a compact subset $\Gamma_\iota\subset V_\iota$, the estimate 
        \[
            |a|_{S^m_f (T^*M), \iota, \Gamma_\iota, N}:=\sum_{|\alpha|+|\beta|\leq N}\sup_{(x, \xi)\in \Gamma_\iota\times \mathbb{R}^n}\jbracket{\xi}^{-m+|\beta|}|\partial_x^\alpha \partial_\xi^\beta \tilde\varphi_{\iota*} a(x, \xi)|<\infty 
        \]
        holds for all $N\in \nni$. 
        \item For $\iota\in I_\infty$, $\Gamma_\iota=[R_\iota, \infty)\times K^\prime_\iota$ where $R_\iota>0$ and $K^\prime_\iota\subset V^\prime_\iota$ is a compact subset, the estimate 
        \begin{align*}
            &|a|_{S^m_f (T^*M), \iota, \Gamma_\iota, N} \\
            &:=\sum_{|\alpha|+|\beta|\leq N}\sup_{(q, p)\in \Gamma_\iota \times \mathbb{R}^n}\frac{f(r)^{-|\alpha^\prime|+|\beta^\prime|}|\partial_q^\alpha \partial_p^\beta \tilde\varphi_{\iota*} a(q, p)|}{\jbracket{\rho \oplus f(r)^{-1}\eta}^{m-|\beta|}} \\
            &<\infty
        \end{align*}
        holds for all $N\in \nni$. 
    \end{itemize}
    We fix a collection of subsets $\{\Gamma_\iota\}_{\iota\in I}$ in the above conditions such that $\bigcup_{\iota\in I} \varphi^{-1}(\Gamma_\iota)=M$. Then we define seminorms on $S^m_f (T^*M)$ as 
    \[
        |a|_{S^m_f (T^*M), N}:=\sum_{\iota\in I} |a|_{S^m_f (T^*M), \iota, \Gamma_\iota, N}. 
    \]
\end{defi}

\begin{rema*}
    Our symbol class $S^m_f (T^*M)$ is equivalent to $S(\jbracket{\rho\oplus f(r)^{-1}\eta}^m, \tilde g_\sigma)$, where 
    \[
    \tilde g_\sigma:=dr^2+f(r)^2\sum_{j=1}^{n-1} (d\theta_j)^2+\jbracket{\rho\oplus f(r)^{-1}\eta}^{-2}\left(d\rho^2+f(r)^{-2}\sum_{j=1}^{n-1} (d\eta_j)^2\right), 
    \]
    in H\"ormander's notation \cite{Hormander85-3}. 
    Since we permit the exponential increase for $f$, the metric $\tilde g_\sigma$ is not necessarily \textit{temperate}, that is, roughly speaking, the coefficients of the metric tensor are at most polynomially increasing. On the theory of pseudodifferential operators on manifolds based on H\"ormander's symbol class, we refer to the works of Baldus \cite{Baldus03a,Baldus03b}. 
    \end{rema*}

As in the case of Euclidean spaces or compact manifolds, we construct a symbol $b(z)$ of the form 
\[
    b(z)(x, \xi)\sim (z-\sigma (P_\hbar)(x, \xi))^{-1}+\hbar b_1 (z)+\cdots
\]
and quantize it to obtain a pseudodifferential operator $\Op^1_{M, \hbar}(b(z))$. We will describe the precise definition of $\Op^1_{M, \hbar}$ in Section \ref{subs_psido_manifolds}. Roughly speaking, we quantize symbols in $S^m_f (T^*M)$ as 
\[
    a \mapsto \frac{1}{(2\pi\hbar)^n}\int_{\mathbb{R}^n} a(r, \theta, \rho, \eta)e^{i\rho (r-r^\prime)/\hbar+i\eta\cdot (\theta-\theta^\prime)/\hbar}\, \diff \rho \diff \eta |\diff r \diff \theta|^{1/2}|\diff r^\prime \diff \theta^\prime|^{1/2}
\]
in polar coordinates. 

Now we state the existence of the parametrix of the resolvent. In the statement below, $C_c^\infty (M; \Omega^{1/2})$ and $L^2(M; \Omega^{1/2})$ are the space of compactly supported smooth half-densities and the $L^2$ space of half-densities on $M$ respectively. 

\begin{theo}\label{theo_parametrix}
Let $P_\hbar \in \mathrm{Diff}^m_{f, \hbar} (M; \Omega^{1/2})$ be an elliptic differential operator. Then we can construct a symbol $b(z)\in S^{-m}_f(T^*M)$ such that 
\[
(z- P_\hbar ) \Op^1_{M, \hbar}(b(z))u=u+ R_\hbar (z)u
\]
for all $u\in C_c^\infty( M; \Omega^{1/2})$ and $z\in\mathbb{C}$ with $\mathrm{dist}(z, \sigma( P_\hbar )(T^* M ))>0$. $ R_\hbar (z)$ are bounded operators on $L^2(M; \Omega^{1/2})$ and 
\[
    \|R_\hbar (z)\|_{L^2\to L^2}=O(\hbar^\infty)
\]
(not uniformly in $z$). 
\end{theo}

We remark that our pseudodifferential operators are not necessarily properly supported. We can find the microlocal analysis on asymptotically hyperbolic manifolds by properly supported pseudodifferential operators in Bouclet \cite{Bouclet11-1} for instance. He used them in order to prove the Strichartz estimates on manifolds with asymptotically hyperbolic ends \cite{Bouclet11-2}.

The difficulty in deducing global properties of our pseudodifferential operators is that it is unknown whether we can obtain a sufficiently rapid off-diagonal decay of the integral kernel of them in the case of $f$ with exponential increasing. Despite this difficulty, we can prove and use the Calder\'on-Vaillancourt type $L^2$ boundedness theorem \cite{Calderon-Vaillancourt71} for example. We prove the $L^2$ boundedness theorem for bounded symbols in Theorem \ref{theo_L2_bdd_simplest} by introducing a scaling. This enables us to avoid the argument of off-diagonal decay of integral kernels. 

Although our definition (pseudo)differential operators are not directly related to Riemannian metrics on $M$, many important examples of differential operators are defined via Riemannian metrics. We will describe a condition of Riemannian metric on $M$ in Assumption \ref{assu_metric_M} in Section \ref{subs_ess_selfadj}. Assumption \ref{assu_metric_M} is related to manifolds with bounded geometry. A complete Riemannian manifold has bounded geometry if and only if it has the positive injectivity radius and its Riemannian curvature tensor and its covariant derivative are globally bounded. Analysis on manifolds with bounded geometry is studied by Engel \cite{Engel18}, Gro{\ss}e and Schneider \cite{Grosse-Schneider13}, and Shubin \cite{Shubin92}. The study by Ammann, Gro{\ss}e and Nistor \cite{Ammann-Grosse-Nistor19} employed the analysis on manifolds with bounded geometry to that on singular domains. 

There are many studies of analysis by embedding the manifold with ends into a compact manifold with boundary and identifying the infinity of the manifold with ends with the boundary of the compact manifold. Melrose \cite{Melrose95} applied this concept to the construction of the geometric scattering theory. On the asymptotic hyperbolic ends in this formulation, we refer to Mazzeo and Melrose \cite{Mazzeo-Melrose87} and Melrose, S\'{a} Baretto and Vasy \cite{Melrose-SaBaretto-Vasy14}. A generalization to the case of variable curvature at infinity is studied in S{\'a} Barreto and Wang \cite{SaBarreto-Wang16}. Another, but related approach is that one construct pseudodifferential calculi on manifolds with a Lie structure at infinity, which are also called Lie manifolds. In this formulation, the non-compact manifold is embedded into a compact manifold with corners. Pseudodifferential calculus on Lie manifolds is studied by Ammann, Lauter and Nistor \cite{Ammann-Lauter-Nistor07} and Nistor \cite{Nistor16}. There is also a formulation on manifolds with fibered corners, which is studied by Debord, Lescure and Rochon \cite{Debord-Lescure-Rochon15}. 
    
Coordinate-free definition of pseudodifferential operators on manifolds is also investigated recently. We can find some of them in Derezi\'nski, Latosi\'nski and Siemssen \cite{Derezinski-Latosinski-Siemssen20}, Levy \cite{Levy} and McKeag and Safarov \cite{McKeag-Safarov11}. In order to realize the coordinate-free definition, they define the parts of $tx+(1-t)y$ and $\xi \cdot (x-y)$ in the definition of pseudodifferential operators 
\[
    u(x) \mapsto \frac{1}{(2\pi)^n}\int_{\mathbb{R}^{2n}} a(tx+(1-t)y, \xi) e^{i\xi \cdot (x-y)}u(y)\, \diff y \diff \xi
\]
in terms of the exponential map which is associated with a connection \cite{Derezinski-Latosinski-Siemssen20,McKeag-Safarov11} or directly equipped the manifold with as in \cite{Levy}. 

This paper is organized as follows. We establish a theory of pseudodifferential operators in polar coordinates in Section \ref{sec_bisymbol}. In Section \ref{sec_ess_selfadj}, we define pseudodifferential operators on manifolds and we investigate a composition of differential operators and pseudodifferential operators. In particular, we prove the main theorem (Theorem \ref{theo_parametrix}) in Section \ref{subs_proof_main} and we investigate the essential self-adjointness of symmetric differential operators in Section \ref{subs_ess_selfadj}.

\section{Pseudodifferential operators in polar coordinates}\label{sec_bisymbol}

\subsection{Definition and fundamental properties}

We introduce a symbol class in local coordinates. 

\begin{defi}\label{defi_symbol_class}
    We define symbol classes $\locsymb{m}f$ by the following condition: 
a smooth function $a\in C^\infty(\mathbb{R}^{2n})$ is in $\locsymb{m}f$ if, for any integer $N\geq 0$, there exists a constant $C>0$ such that 
\[
|(f(r)^{-1}\partial_\theta)^{\alpha^\prime} (f(r)\partial_\eta)^{\beta^\prime} \partial_r^{\alpha_0} \partial_\rho^{\beta_0}a(q, p)|\leq C\jbracket{\rho\oplus f(r)^{-1}\eta}^{m-|\beta|}
\]
for all $(q, p)=\mathbb{R}^{2n}$ and multiindices $\alpha, \beta\in \nni^n$ with $|\alpha|+|\beta|\leq N$. 
The seminorms $|a|_{\locsymb{m}{f}, N}$ are defined as
\begin{align*}
|a|_{\locsymb{m}{f}, N} 
:=
\sum_{|\alpha|+|\beta|\leq N}\left\|\frac{(f(r)^{-1}\partial_\theta)^{\alpha^\prime} (f(r)\partial_\eta)^{\beta^\prime} \partial_r^{\alpha_0} \partial_\rho^{\beta_0}a}{\jbracket{\rho\oplus f(r)^{-1}\eta}^{m-|\beta|}}\right\|_{L^\infty(\mathbb{R}^{2n})}. 
\end{align*}
\end{defi}

Next we define pseudodifferential operators with symbols in $\locsymb{m}f$. We regard the Euclidean space $\mathbb{R}^n$ as a product of radial variable $r\in \mathbb{R}$ and angular variable $\theta\in \mathbb{R}^{n-1}$. 

\begin{defi}
    Fix parameters $t\in [0, 1]$ and $m\in \mathbb{R}$. For a symbol $a\in \locsymb{m}f$, we define a pseudodifferential operator $\Op^t_\hbar (a)$ with symbol $a$ as  
\begin{align*}
&\Op^t_\hbar (a)(v(q)|\diff q|^{1/2}) \\
&:=\frac{1}{(2\pi\hbar)^n}\int_{\mathbb{R}^{2n}}a\left(tq+(1-t)q^\prime, p\right)e^{ip\cdot (q-q^\prime)/\hbar}v(q^\prime)\, \diff q^\prime \diff p |\diff q|^{1/2}. 
\end{align*}
\end{defi}

We denote the class of smooth half-densities by $C^\infty (M; \Omega^{1/2})$ and that of compactly supported smooth half-densities by $C_c^\infty (M; \Omega^{1/2})$. We first prove the smoothness of $\mathrm{Op}^t_\hbar (a)u$ for $u\in C_c^\infty(\mathbb{R}^n; \Omega^{1/2})$. 

\begin{prop}\label{prop_smoothness}
If $a\in \locsymb{m}{f}$, then $\Op^t_\hbar  (a)$ defines a continuous linear operator from $C_c^\infty(\mathbb{R}^n; \Omega^{1/2})$ to $C^\infty(\mathbb{R}^n; \Omega^{1/2})$. 
\end{prop}

\begin{proof}
    Let $u=v|\diff q|^{1/2}\in C_c^\infty (\mathbb{R}^n; \Omega^{1/2})$. We take an arbitrary relatively compact subset $K\subset \mathbb{R}^n$ and restrict $q=(r, \theta)$ to $K$. Noting $u\in C_c^\infty (\mathbb{R}^n; \Omega^{1/2})$, we apply integration by parts for the integration by $q^\prime$ employing $(1+ip\cdot \partial_{q^\prime})/\jbracket{p}^2$. Then, for any integer $N\geq 0$, there exists a constant $C>0$ and an integer $N^\prime$ such that the estimate 
    \begin{equation}\label{eq_d_to_e}
        \left|\int_{\mathbb{R}^n} a(tq+(1-t)q^\prime, p)e^{-ip\cdot q^\prime/\hbar}v(q^\prime)\, \diff q^\prime\right| 
        \leq C \jbracket{p}^{-N}\sum_{|\alpha|\leq N^\prime} \| \partial_q^\alpha v\|_{L^\infty(\mathbb{R}^n)}
    \end{equation}
    holds for all $q\in K$. Thus we can apply differentiation under integral and obtain the smoothness of $\Op^t_\hbar  (a)u$. The continuity of $\Op^t_\hbar  (a)$ also follows from \eqref{eq_d_to_e}. 
\end{proof}

Noting that $\locsymb{m}{f}$ is included in the space of tempered distributions $\mathscr{S}^\prime (\mathbb{R}^{2n})$ by $\inf f(r)\geq 1$, we state a continuity of expectation value. 

\begin{prop}\label{prop_expect}
    Fix compactly supported half-densities $u, v\in C_c^\infty (M; \Omega^{1/2})$. Then, if a sequence $\{ a_j\}_{j=1}^\infty \subset \locsymb{m}{f}$ converges to a symbol $a\in \locsymb{m}{f}$ in $\mathscr{S}^\prime (\mathbb{R}^{2n})$, then we have 
    \begin{equation}\label{eq_limit_expectation}
        \lim_{j\to \infty}\jbracket{ \Op^t_\hbar  (a_j)u, v}=\jbracket{\Op^t_\hbar  (a)u, v}. 
    \end{equation}
\end{prop}

\begin{proof}
    We set  $u=\tilde u |\diff q|^{1/2}$ and $v=\tilde v |\diff q|^{1/2}$. By changing the order of integrals, the inner product $\jbracket{\Op^t_\hbar  (a)u, v}$ is calculated as 
    \begin{equation}\label{eq_distribution_quantization}
        \jbracket{\Op^t_\hbar  (a)u, v}=\int_{\mathbb{R}^{2n}} a(q, p) \varphi (q, p)\, \diff q \diff p
    \end{equation}
    where 
    \[
        \varphi(q, p):=
        \frac{1}{(2\pi \hbar (1-t))^n} \int_{\mathbb{R}^n} e^{i\xi\cdot (q^\prime -q)/\hbar (1-t)} \overline{\tilde v(q^\prime)}\tilde u\left( \frac{-tq^\prime+q}{1-t}\right)
    \]
    if $0\leq t<1$ and 
    \[
        \varphi(q, p):=\frac{1}{(2\pi\hbar)^n}\overline{\tilde v(x)}e^{i\xi\cdot x/\hbar}\int_{\mathbb{R}^n} \tilde u(y)e^{-i\xi\cdot y/\hbar}\, \diff y
    \]
    if $t=1$. In both cases, $\varphi$ is a rapidly decreasing function of $(q, p)\in \mathbb{R}^{2n}$. Thus the assumption $a_j\to a$ in $\mathscr{S}^\prime (\mathbb{R}^{2n})$ implies \eqref{eq_limit_expectation}. 
\end{proof}

\begin{rema*}
    Proposition \ref{prop_expect} holds for general $\{a_j \}_{j=1}^\infty \subset \mathscr{S}^\prime (\mathbb{R}^{2n})$ and $a\in \mathscr{S}^\prime (\mathbb{R}^{2n})$ if we define $\Op^t_\hbar  (a)$ for $a\in \mathscr{S}^\prime (\mathbb{R}^{2n})$ by \eqref{eq_distribution_quantization}. Moreover, we can take rapidly decreasing half-densities $u$ and $v$ in the sense that $\tilde u$ and $\tilde v$ defined by $u=\tilde u |\diff q|^{1/2}$ and $v=\tilde v |\diff q|^{1/2}$ are rapidly decreasing functions. 
\end{rema*}

\begin{prop}
    \label{prop_bisymbol_symbol}
    Let $B\locsymb{m}{f, t}$ be a class of smooth functions $a(q, p, q^\prime)\in C^\infty (\mathbb{R}^{3n})$ with $|a|_{B\locsymb{m}{f, t}, N}<\infty$ for all $N\in \nni$, where 
    \begin{align*}
        &|a|_{B\locsymb{m}{f, t}, N} \\
        &:=\sum_{|\alpha|+|\beta|+|\gamma|\leq N} 
        \| \jbracket{\rho\oplus f(tr+(1-t)r^\prime)^{-1}\eta}^{-m+|\beta|} \\
        &\quad \times f(tr+(1-t)r^\prime)^{-|\alpha^\prime|+|\beta^\prime|-|\gamma^\prime|}\partial_q^\alpha \partial_p^\beta \partial_{q^\prime}^\gamma a\|_{L^\infty (\mathbb{R}^{3n})}. 
    \end{align*}
    For $a\in B\locsymb{m}{f, t}$, we define an operator $\Op_\hbar (a): C_c^\infty (\mathbb{R}^n; \Omega^{1/2})\to C^\infty (\mathbb{R}^n; \Omega^{1/2})$ as 
    \[
        \Op_\hbar (a)(v|\diff q|^{1/2}):=\frac{1}{(2\pi\hbar)^n}\int_{\mathbb{R}^{2n}} a(q, p, q^\prime)e^{ip\cdot (q-q^\prime)/\hbar}v(q^\prime)\, \diff q^\prime \diff p |\diff q|^{1/2}
    \]
    Then, there exists a continuous linear mapping $a\in B\locsymb{m}{f, t}\mapsto b^t\in \locsymb{m}{f}$ such that the equality 
    \begin{equation}
        \label{eq_bisymb_symb_op}
        \Op_\hbar (a)=\Op^t_\hbar  (b^t)
    \end{equation}
    holds. 

    Moreover, $b^t$ has an asymptotic expansion 
    \[
        b^t(q, p)
        \sim \sum_{j=0}^\infty \frac{\hbar^j}{j !}(\partial_p\cdot \partial_{q^\prime})^j|_{q^\prime=0}a(q+(1-t)q^\prime, p, q-tq^\prime)
    \]
    in the sense of 
    \[
        b^t(q, p)
        - \sum_{j=0}^N \frac{\hbar^j}{j !}(\partial_p\cdot \partial_{q^\prime})^j|_{q^\prime=0}a(q+(1-t)q^\prime, p, q-tq^\prime)
        \in \hbar^{N+1}\locsymb{m-N-1}{f}
    \]
    for all $N\in \nni$. 
\end{prop}

\begin{proof}
    \fstep{Definition of $\bm{b^t}$}If $a\in C_c^\infty (\mathbb{R}^{3n})$, then $b_t$ is represented as an explicit formula 
    \begin{equation}
        \label{eq_bisymb_symb}
        b^t (q, p)=\frac{1}{(2\pi\hbar)^n}\int_{\mathbb{R}^{2n}} a(q+(1-t)q^\prime, p+p^\prime, q-tq^\prime)e^{-ip^\prime\cdot q^\prime/\hbar} \diff q^\prime \diff p^\prime. 
    \end{equation}
    We extend the equation \eqref{eq_bisymb_symb} to a general symbol $a\in B\locsymb{m}{f}$ in the following. Take a partition of unity $\{\psi_j=\psi(\cdot -j)\}_{j\in\mathbb{Z}}$ with $\rmop{supp}\psi \subset (-1, 1)$ and a cutoff function $\chi (q)\in C_c^\infty (\mathbb{R}^n)$ with $\chi(q)=1$ if $|q|\leq 1$. For a small quantity $\varepsilon>0$, We define functions $a^t_{jk}(q, p, q^\prime)$ as 
    \begin{equation}\label{eq_atjk}
    \begin{aligned}
        a^t_{jk, \varepsilon}(q, p, q^\prime)
        :=&
        \psi_j (r+(1-t)r^\prime)\psi_k (r-t r^\prime)\chi (\varepsilon q)\chi(\varepsilon p)\chi(\varepsilon q^\prime)  \\
        &\times a((\Theta^t_{jk})^{-1}(q+(1-t)q^\prime), \Theta^t_{jk}(p), (\Theta^t_{jk})^{-1}(q-t q^\prime))
    \end{aligned}
\end{equation}
    where 
    \[
        \Theta^t_{jk}(r, \theta):=(r, f(tj+(1-t)k)\theta). 
    \]
    We define $b^t$ as 
    \begin{equation}
        \label{eq_defi_bt}
        b^t(q, p):=\sum_{j, k\in \mathbb{Z}} b^t_{jk}(q, p)
    \end{equation}
    where 
    \[
        b^t_{jk} (q, p):=\lim_{\varepsilon\to +0}b^t_{jk, \varepsilon} (q, p)
    \]
    and 
    \begin{equation}
        \label{eq_defi_bte}
        b^t_{jk, \varepsilon}(q, p):=
        \frac{1}{(2\pi\hbar)^n}\int_{\mathbb{R}^{2n}} a^t_{jk, \varepsilon}(\Theta^t_{jk}(q), (\Theta^t_{jk})^{-1}(p)+p^\prime, q^\prime) 
        e^{-ip^\prime\cdot q^\prime/\hbar} \diff q^\prime \diff p^\prime. 
    \end{equation}    

    \fstep{Well-definedness of $\bm{b^t_{jk}}$}We prove the well-definedness of \eqref{eq_defi_bt} and \eqref{eq_defi_bte}. 
    We consider derivatives of the pushforward ${\widetilde\Theta}^t_{jk*}b^t_{jk, \varepsilon}$ by the canonical mapping ${\widetilde\Theta}^t_{jk}$ associated with $\Theta^t_{jk}$. They are calculated as 
    \begin{equation}\label{eq_btjk_push_der}
    \begin{aligned}
        \partial_q^\alpha \partial_p^\beta {\widetilde\Theta}^t_{jk*}b^t_{jk, \varepsilon}(q, p)
        =\frac{1}{(2\pi\hbar)^n} 
        \int_{\mathbb{R}^{2n}} (\partial_q^\alpha \partial_p^\beta a^t_{jk, \varepsilon})(q, p+p^\prime, q^\prime) 
        e^{-ip^\prime\cdot q^\prime/\hbar} \diff q^\prime \diff p^\prime. 
    \end{aligned}
\end{equation}
    Integration by parts by the differential operator
    \begin{equation}\label{eq_defi_l}
        L:=\frac{1+i\hbar p^\prime \cdot \partial_{q^\prime}+i\hbar q^\prime \cdot\partial_{p^\prime}}{1+|q^\prime|^2+|p^\prime|^2}
    \end{equation}
    yields the equation
    \begin{equation}\label{eq_btjk}
    \begin{aligned}
        \partial_q^\alpha \partial_p^\beta {\widetilde\Theta}^t_{jk*}b^t_{jk, \varepsilon} (q, p)
        =\frac{1}{(2\pi\hbar)^n}
        \int_{\mathbb{R}^{2n}} (\transp{L})^N(\partial_q^\alpha \partial_p^\beta a^t_{jk, \varepsilon})(q, p+p^\prime, q^\prime) e^{-ip^\prime\cdot q^\prime/\hbar} \diff q^\prime \diff p^\prime
    \end{aligned}
\end{equation}
    for an arbitrary $N\in \nni$. 

    We set $f^t_{jk}:=f(tj+(1-t)k)$. Since $|r-tj-(1-t)k|\leq 1$ on $\rmop{supp}a^t_{jk}$, we have the estimate 
    \begin{align*}
        &|\partial_q^\alpha \partial_p^\beta \partial_{q^\prime}^\gamma a^t_{jk, \varepsilon}(q, p, q^\prime)| \\
        &\leq C\left(\frac{f^t_{jk}}{f(r)}\right)^{-|\alpha^\prime|+|\beta^\prime|-|\gamma^\prime|} |a|_{B\locsymb{m}{f, t}, N} 1_{[j-1, j+1]}(r+(1-t)r^\prime)1_{[k-1, k+1]}(r-tr^\prime) \\
        &\quad\times \jbracket{\rho \oplus (f^t_{jk}/f(r))\eta}^{m-|\beta|} \\
        &\leq C|a|_{B\locsymb{m}{f, t}, N}1_{[-1, 1]}(r-tj-(1-t)k)1_{[-2, 2]}(r^\prime-j+k)\jbracket{p}^{m-|\beta|}
    \end{align*}
    for a constant $C=C_N>0$ independent of $(q, p, q^\prime)\in \mathbb{R}^{3n}$, $j, k\in \mathbb{Z}$ and $|\alpha|+|\beta|+|\gamma|\leq N$. Thus the integrand in the right hand side of \eqref{eq_btjk} is estimated as 
    \[
        |(\transp{L})^N(\partial_q^\alpha \partial_p^\beta a^t_{jk, \varepsilon})(q, p+p^\prime, q^\prime)|
        \leq C\jbracket{q^\prime \oplus p^\prime}^{-N}\jbracket{p+p^\prime}^{m-|\beta|}
    \]
    uniformly in $\varepsilon>0$. Hence we can apply the Lebesgue dominated convergence theorem to \eqref{eq_btjk} and obtain 
    \begin{equation}
        \label{eq_btjk_2}
        \begin{aligned}
         &   \lim_{\varepsilon\to +0}\partial_q^\alpha \partial_p^\beta {\widetilde\Theta}^t_{jk*}b^t_{jk, \varepsilon} (q, p) \\
        &=\frac{1}{(2\pi\hbar)^n}
        \int_{\mathbb{R}^{2n}} (\transp{L})^N(\partial_q^\alpha \partial_p^\beta a^t_{jk})(q, p+p^\prime, q^\prime) e^{-ip^\prime\cdot q^\prime/\hbar} \diff q^\prime \diff p^\prime
        \end{aligned}
    \end{equation}
    where 
        \begin{align*}
            a^t_{jk}(q, p, q^\prime)
            :=&\psi_j (r+(1-t)r^\prime)\psi_k (r-t r^\prime)  \\
            &\times a((\Theta^t_{jk})^{-1}(q+(1-t)q^\prime), \Theta^t_{jk}(p), (\Theta^t_{jk})^{-1}(q-t q^\prime)). 
        \end{align*}

    Next we show that the convergence 
    \begin{equation}\label{eq_uniform_weighted}
        \lim_{\varepsilon\to +0}\jbracket{q\oplus p}^{-N}\partial_q^\alpha \partial_p^\beta {\widetilde\Theta}^t_{jk*}b^t_{jk, \varepsilon} (q, p)
        =\jbracket{q\oplus p}^{-N}\partial_q^\alpha \partial_p^\beta {\widetilde\Theta}^t_{jk*}b^t_{jk} (q, p)
    \end{equation}
    is uniform in $(q, p)\in \mathbb{R}^{2n}$ if we take a sufficiently large $N\in \nni$. We have the estimate 
    \begin{equation}\label{eq_epsilon_diff}
    \begin{aligned}
        &\left|\int_{\mathbb{R}^{2n}} (\transp{L})^N(\partial_q^\alpha \partial_p^\beta (a^t_{jk}-a^t_{jk, \varepsilon}))(q, p+p^\prime, q^\prime) e^{-ip^\prime\cdot q^\prime/\hbar} \diff q^\prime \diff p^\prime\right| \\
        &\leq C\int_{\mathbb{R}^{2n}}\jbracket{q^\prime \oplus p^\prime}^{-N}\jbracket{p+p^\prime}^{m-|\beta|} \\
        &\quad\times (1_{[\varepsilon^{-1}, \infty)}(|q|)+1_{[\varepsilon^{-1}, \infty)}(|p+p^\prime|)+1_{[\varepsilon^{-1}, \infty)}(|q^\prime|))\, \diff q^\prime \diff p^\prime. 
    \end{aligned}
\end{equation}
    Each integrals are estimated as 
    \[
        \int_{\mathbb{R}^{2n}}\jbracket{q^\prime \oplus p^\prime}^{-N}\jbracket{p+p^\prime}^{m-|\beta|}1_{[\varepsilon^{-1}, \infty)}(|q|)\, \diff q^\prime \diff p^\prime 
        \leq C\varepsilon \jbracket{q}\jbracket{p}^{m-|\beta|}, 
    \]
    \[
        \int_{\mathbb{R}^{2n}}\jbracket{q^\prime \oplus p^\prime}^{-N}\jbracket{p+p^\prime}^{m-|\beta|}1_{[\varepsilon^{-1}, \infty)}(|p+p^\prime|)\, \diff q^\prime \diff p^\prime 
        \leq C\varepsilon\jbracket{p}^{m-|\beta|+1}
    \]
    and 
    \[
        \int_{\mathbb{R}^{2n}}\jbracket{q^\prime \oplus p^\prime}^{-N}\jbracket{p+p^\prime}^{m-|\beta|}1_{[\varepsilon^{-1}, \infty)}(|q^\prime|)\, \diff q^\prime \diff p^\prime 
        \leq C\varepsilon\jbracket{p}^{m-|\beta|}. 
    \]
    Thus \eqref{eq_epsilon_diff} becomes 
    \begin{align*}
        |\partial_q^\alpha \partial_p^\beta {\widetilde\Theta}^t_{jk*}(b^t_{jk, \varepsilon}-b^t_{jk}) (q, p)|
        \leq C\hbar^{-n}\varepsilon \jbracket{q}\jbracket{p}^{m-|\beta|+1}. 
    \end{align*}
    This inequality shows that the convergence \eqref{eq_uniform_weighted} is uniform in $(q, p)\in \mathbb{R}^{2n}$ by taking $N=m-|\beta|+2$. 

    Hence we can change the order of the limit and the differentiation in \eqref{eq_btjk_2} and obtain 
    \begin{equation}
        \label{eq_btjk_3}
        \begin{aligned}
            &\partial_q^\alpha \partial_p^\beta {\widetilde\Theta}^t_{jk*}b^t_{jk} (q, p) \\
        &=\frac{1}{(2\pi\hbar)^n}
        \int_{\mathbb{R}^{2n}} (\transp{L})^N(\partial_q^\alpha \partial_p^\beta a^t_{jk})(q, p+p^\prime, q^\prime) e^{-ip^\prime\cdot q^\prime/\hbar} \diff q^\prime \diff p^\prime.  
        \end{aligned}
    \end{equation}

    \fstep{Well-definedness of $\bm{b^t}$}
    By \eqref{eq_btjk_3}, the function $\partial_q^\alpha \partial_p^\beta {\widetilde\Theta}^t_{jk*}b^t_{jk} (q, p)$ is estimated as 
    \begin{align*}
        &|\partial_q^\alpha \partial_p^\beta {\widetilde\Theta}^t_{jk*}b^t_{jk} (q, p)| \\
        &\leq C|a|_{B\locsymb{m}{f, t}, N+|\alpha|+|\beta|}1_{[-1, 1]}(r-tj-(1-t)k) \\
        &\quad\times\int_{\mathbb{R}^{2n}}1_{[-2, 2]}(r^\prime-j+k)\jbracket{p+p^\prime}^{m-|\beta|}\jbracket{q^\prime\oplus p^\prime}^{-N}\, \diff q^\prime \diff p^\prime \\
        &\leq C|a|_{B\locsymb{m}{f, t}, N+|\alpha|+|\beta|}1_{[-1, 1]}(r-tj-(1-t)k)\jbracket{j-k}^{-N}\jbracket{p}^{m-|\beta|}. 
    \end{align*}
    Hence the derivatives of the summand of the right hand side of \eqref{eq_defi_bt} is estimated as 
    \begin{align*}
        &|\partial_q^\alpha \partial_p^\beta b^t_{jk} (q, p)| \\
        &\leq C(f^t_{jk})^{|\alpha^\prime|-|\beta^\prime|}|a|_{B\locsymb{m}{f, t}, N+|\alpha|+|\beta|}1_{[-1, 1]}(r-tj-(1-t)k) \\
        &\quad \times \jbracket{j-k}^{-N}\jbracket{\rho\oplus (f^t_{jk})^{-1}\eta}^{m-|\beta|} \\
        &\leq Cf(r)^{|\alpha^\prime|-|\beta^\prime|}|a|_{B\locsymb{m}{f, t}, N+|\alpha|+|\beta|}1_{[-1, 1]}(r-tj-(1-t)k) \\
        &\quad \times \jbracket{j-k}^{-N}\jbracket{\rho\oplus f(r)^{-1}\eta}^{m-|\beta|} \\
        &\leq Cf(r)^{|\alpha^\prime|-|\beta^\prime|}|a|_{B\locsymb{m}{f, t}, N+|\alpha|+|\beta|}\jbracket{r}^N\jbracket{tj-(1-t)k}^{-N} \\
        &\quad \times \jbracket{j-k}^{-N}\jbracket{\rho\oplus f(r)^{-1}\eta}^{m-|\beta|}. 
    \end{align*}
    This implies that the sum  
    \[
        \sum_{j, k\in \mathbb{Z}} \jbracket{r}^{-N}f(r)^{-|\alpha^\prime|+|\beta^\prime|}\jbracket{\rho\oplus f(r)^{-1}\eta}^{-m+|\beta|}|\partial_q^\alpha \partial_p^\beta b^t_{jk} (q, p)|
    \]
    converges uniformly in $(q, p)\in \mathbb{R}^{2n}$ by the Weierstrass $M$-test. Thus we can change the order of the summation and the derivative in \eqref{eq_defi_bt} and obtain 
    \[
        \partial_q^\alpha \partial_p^\beta b^t (q, p)
        =\sum_{j, k\in \mathbb{Z}} \partial_q^\alpha \partial_p^\beta b^t_{jk} (q, p)
    \]
    and 
    \[
        |\partial_q^\alpha \partial_p^\beta b^t (q, p)|
        \leq Cf(r)^{|\alpha^\prime|-|\beta^\prime|}\jbracket{\rho\oplus f(r)^{-1}\eta}^{m-|\beta|}|a|_{B\locsymb{m}{f, t}, N+|\alpha|+|\beta|}. 
    \]
    This estimate shows the continuity of the linear mapping $a\in B\locsymb{m}{f, t}\mapsto b^t \in \locsymb{m}{f}$. 

    \fstep{Proof of \eqref{eq_bisymb_symb_op}}Since $a_{jk, \varepsilon}\in C_c^\infty (\mathbb{R}^{3n})$, we have the equation 
    \begin{equation}
        \label{eq_bisymb_symb_jke}
        \jbracket{\psi_j\Op_\hbar (c^t_{jk, \varepsilon})\psi_k u, v}=\jbracket{ \Op^t_\hbar  (b^t_{jk, \varepsilon})u, v}
    \end{equation}
    where
    \[
        c^t_{jk, \varepsilon} (q, p, q^\prime)
        =a(q, p, q^\prime)\chi (\varepsilon \Theta^t_{jk}(tq+(1-t)q^\prime))
        \chi (\varepsilon (\Theta^t_{jk})^{-1}(p)) \chi (\varepsilon \Theta^t_{jk}(q-q^\prime)). 
    \]
    Since the uniform convergence \eqref{eq_epsilon_diff} implies the convergence $b^t_{jk, \varepsilon}\to b^t_{jk}$ in $\mathscr{S}^\prime(\mathbb{R}^{2n})$, we can apply Proposition \ref{prop_expect} to \eqref{eq_bisymb_symb_jke} and obtain 
    \begin{equation}
        \label{eq_bisymb_symb_jk}
        \jbracket{\psi_j \Op_\hbar (a)\psi_k u, v}=\jbracket{ \Op^t_\hbar  (b^t_{jk})u, v}. 
    \end{equation}
    Finally, since $u$ and $v$ are compactly supported, the sum of \eqref{eq_bisymb_symb_jk} over $j, k\in \mathbb{Z}$ becomes a finite sum. Thus we obtain the desired equation \eqref{eq_bisymb_symb_op}. 

    \fstep{Asymptotic expansion}We expand \eqref{eq_btjk_push_der} with respect to $p^\prime$: 
    \begin{align}
        &\partial_q^\alpha \partial_p^\beta {\widetilde\Theta}^t_{jk*}b^t_{jk, \varepsilon}(q, p) \nonumber\\
        &=\frac{1}{(2\pi\hbar)^n} 
        \int_{\mathbb{R}^{2n}} \biggl(\sum_{|\gamma|\leq N}\frac{1}{\gamma !}(\partial_q^\alpha \partial_p^{\beta+\gamma} a^t_{jk, \varepsilon})(q, p, q^\prime)(p^\prime)^\gamma \nonumber\\
        &\quad+\sum_{|\gamma|=N+1} R^t_{jk, \alpha\beta\gamma, \varepsilon}(q, p, q^\prime, p^\prime)(p^\prime)^\gamma \biggr)
        e^{-ip^\prime\cdot q^\prime/\hbar} \diff q^\prime \diff p^\prime \nonumber\\
        &\begin{aligned}
        &=\sum_{|\gamma|\leq N}\frac{(i^{-1}\hbar)^{|\gamma|}}{\gamma !}(\partial_q^\alpha \partial_p^{\beta+\gamma}\partial_{q^\prime}^\gamma a^t_{jk, \varepsilon})(q, p, 0) \\
        &\quad +\frac{(i^{-1}\hbar)^{N+1}}{(2\pi\hbar)^n} 
        \sum_{|\gamma|=N+1}\int_{\mathbb{R}^{2n}} \partial_{q^\prime}^\gamma R^t_{jk, \alpha\beta\gamma, \varepsilon}(q, p, q^\prime, p^\prime)
        e^{-ip^\prime\cdot q^\prime/\hbar} \diff q^\prime \diff p^\prime. 
        \end{aligned}\label{eq_btjk_push_der_expand}
    \end{align}
    where 
    \[
        R^t_{jk, \alpha\beta\gamma, \varepsilon}(q, p, q^\prime, p^\prime)
        :=\frac{N+1}{\gamma !}\int_0^1 (\partial_q^\alpha \partial_p^{\beta+\gamma} a^t_{jk, \varepsilon})(q, p+\sigma p^\prime, q^\prime)\, \diff \sigma. 
    \]
    $R^t_{jk, \alpha\beta\gamma, \varepsilon}$ has an estimate 
    \begin{align*}
        &|(\transp{L})^{2n+1}\partial_{q^\prime}^\gamma R^t_{jk, \alpha\beta\gamma, \varepsilon}(q, p, q^\prime, p^\prime)| \\
        &\leq C\jbracket{p}^{m-|\beta|-|\gamma|}\jbracket{q^\prime \oplus p^\prime}^{-N^\prime}1_{[-1, 1]}(r-tj-(1-t)k)1_{[-2, 2]}(r^\prime-j+k) \\
        &\quad \times \sum_{|\delta|\leq N^\prime}\| \jbracket{p}^{-m+|\beta|+|\gamma|}\partial_{q, p, q^\prime}^\delta \partial_q^\alpha \partial_p^{\beta+\gamma}\partial_{q^\prime}^\gamma a^t_{jk}\|_{L^\infty} \\
        &\leq C\jbracket{p}^{m-|\beta|-|\gamma|}\jbracket{q^\prime \oplus p^\prime}^{-2n-1}1_{[-1, 1]}(r-tj-(1-t)k)1_{[-2, 2]}(r^\prime-j+k) \\
        &\quad \times |a|_{B\locsymb{m}{f, t}, |\alpha|+|\beta|+2|\gamma|+N^\prime}
    \end{align*}
    which is uniform in $j, k\in \mathbb{Z}$. Here $L$ is the operator defined in \eqref{eq_defi_l} and $N^\prime \in \nni$ is a sufficiently large integer. Thus, as in the proof of \eqref{eq_btjk_2}, we take a limit $\varepsilon\to +0$ of \eqref{eq_btjk_push_der_expand} and obtain 
    \begin{align*}
        &\partial_q^\alpha \partial_p^\beta {\widetilde\Theta}^t_{jk*}b^t_{jk}(q, p) \\
        &=\sum_{|\gamma|\leq N}\frac{(i^{-1}\hbar)^{|\gamma|}}{\gamma !}(\partial_q^\alpha \partial_p^{\beta+\gamma}\partial_{q^\prime}^\gamma a^t_{jk})(q, p, 0) \\
        &\quad +\frac{(i^{-1}\hbar)^{N+1}}{(2\pi\hbar)^n} 
        \sum_{|\gamma|=N+1}\int_{\mathbb{R}^{2n}} (\transp{L})^{N^\prime}\partial_{q^\prime}^\gamma R^t_{jk, \alpha\beta\gamma}(q, p, q^\prime, p^\prime)
        e^{-ip^\prime\cdot q^\prime/\hbar} \diff q^\prime \diff p^\prime. 
    \end{align*}
    Here 
    \[
        R^t_{jk, \alpha\beta\gamma}(q, p, q^\prime, p^\prime)
        :=\lim_{\varepsilon\to +0}R^t_{jk, \alpha\beta\gamma, \varepsilon}(q, p, q^\prime, p^\prime). 
    \]
    Thus we have the estimate of the error term  
    \begin{align*}
        &\left|\partial_q^\alpha \partial_p^\beta {\widetilde\Theta}^t_{jk*}b^t_{jk}(q, p)-\sum_{|\gamma|\leq N}\frac{(i^{-1}\hbar)^{|\gamma|}}{\gamma !}(\partial_q^\alpha \partial_p^{\beta+\gamma}\partial_{q^\prime}^\gamma a^t_{jk})(q, p, 0)\right| \\
        &\leq C\hbar^{N+1-n}\jbracket{p}^{m-|\beta|-N-1}
        1_{[-1, 1]}(r-tj-(1-t)k)\jbracket{j-k}^{-n-1} \\
        &\quad \times |a|_{B\locsymb{m}{f, t}, |\alpha|+|\beta|+2N+N^\prime+2}. 
    \end{align*}
    This is equivalent to 
    \begin{align*}
        &\left|\partial_q^\alpha \partial_p^\beta b^t_{jk}(q, p)-\sum_{l=0}^N \hbar^l\partial_q^\alpha \partial_p^\beta d^t_{jk, l}(q, p)\right| \\
        &\leq C\hbar^{N+1-n}\jbracket{\rho\oplus f(r)^{-1}\eta}^{m-|\beta|-N-1}
        1_{[-1, 1]}(r-tj-(1-t)k)\jbracket{j-k}^{-n-1} \\
        &\quad\times |a|_{B\locsymb{m}{f, t}, |\alpha|+|\beta|+2|\gamma|+N^\prime} 
    \end{align*}
    where 
    \[
        d^t_{jk, l}(q, p):=\frac{1}{l !}(i\partial_p \cdot \partial_{q^\prime})^l|_{q^\prime=0}(\psi_j(r+(1-t)r^\prime)\psi_k (r-tr^\prime)a(q+(1-t)q^\prime, p, q-tq^\prime)). 
    \]
    By the same method in the proof of the convergence of $\sum_{j, k}b^t_{jk}$, the sum 
    \[
        \sum_{j, k\in \mathbb{Z}} \jbracket{r}^{-N^\prime}f(r)^{-|\alpha^\prime|+|\beta^\prime|}\jbracket{\rho\oplus f(r)^{-1}\eta}^{-m+|\beta|+|\gamma|}|\partial_q^\alpha \partial_p^\beta d^t_{jk, l} (q, p)|
    \]
    converges uniformly in $(q, p)\in \mathbb{R}^{2n}$ if we take an integer $N^\prime \in \nni$ sufficiently large. Hence we obtain the desired estimate
    \begin{align*}
        &\left|\partial_q^\alpha \partial_p^\beta b^t(q, p)-\sum_{l=0}^N \frac{\hbar^l}{l !}\partial_q^\alpha \partial_p^\beta (i\partial_p\cdot \partial_{q^\prime})^l|_{q^\prime=0}a(q+(1-t)q^\prime, p, q-tq^\prime)\right| \\
        &\leq C\hbar^{N+1-n}\jbracket{\rho\oplus f(r)^{-1}\eta}^{m-|\beta|-N-1}
         |a|_{B\locsymb{m}{f, t}, |\alpha|+|\beta|+2|\gamma|+N^\prime}
    \end{align*}
    replacing $N$ to $N+n$. 
\end{proof}

\subsection{Boundedness on $L^2$ space}

In this section, we prove the following Calder\'on-Vaillancourt type theorem. We define the space $L^2(M; \Omega^{1/2})$ as the completion of $C_c^\infty (M; \Omega^{1/2})$ with respect to the inner product 
\[
    \jbracket{u, v}_{L^2(M; \Omega^{1/2})}:=\int_M u \overline{v}. 
\]

\begin{theo}[$L^2$ boundedness]\label{theo_L2_bdd_simplest}
    There exist  constants $C>0$ and $N\geq 0$ such that the inequality
    \[
    \|\Op^t_\hbar  (a)u\|_{L^2(\mathbb{R}^n; \Omega^{1/2})}\leq C|a|_{\locsymb{0}{f}, N}\|u\|_{L^2(\mathbb{R}^n; \Omega^{1/2})}
    \]
    holds for all $a\in \locsymb{0}{f}$, all $u\in C_c^\infty(\mathbb{R}^n)$ and all $t\in [0, 1]$. 
    \end{theo}

Let $\{\psi_j=\psi(\cdot-j)\}_{j\in \mathbb{Z}}$ be a partition of unity of $\mathbb{R}$ where $\psi\in C_c^\infty((-1, 1))$,  $\psi\geq 0$. The multiplication operator by the function $\psi_j(r)$ is denoted as $\psi_j$ too: 
\[
(\psi_ju)(r, \theta):=\psi_j(r)u(r, \theta). 
\]

The important step for proving $L^2$ boundedness of $\Op_\hbar (a)$ is an estimate of the $L^2$ operator norm of $\psi_j\Op_\hbar (a)\psi_k$. 

\begin{prop}\label{prop_disjoint_union}
For any $N\geq 0$, there exists a constant $C>0$ and an integer $N^\prime>0$ such that 
\[
\|\psi_j\Op^t_\hbar  (a)\psi_ku\|_{L^2(\mathbb{R}^n; \Omega^{1/2})}\leq C|a|_{\locsymb{m}{f}, N^\prime}\jbracket{j-k}^{-N}\|u\|_{L^2(\mathbb{R}^n; \Omega^{1/2})}
\]
holds for all $t\in [0, 1]$, $j$, $k\in\mathbb{Z}$, $a\in \locsymb{m}{f}$ and $u\in C_c^\infty(\mathbb{R}^n)$. 
\end{prop}

In order to prove Proposition \ref{prop_disjoint_union}, we employ a kind of scaling arguments. We define a linear diffeomorphism $\Theta^t_{jk}:\mathbb{R}^n \to \mathbb{R}^n$ as 
\[
    \Theta^t_{jk}(r, \theta):=(r, f^t_{jk}\theta)
\]
where
\[
f^t_{jk}:=f(tj+(1-t)k). 
\]

The scaling operator is the pull-back by $\Theta^t_{jk}$: 
\[
(\Theta^t_{jk})^*(v|\diff q|^{1/2})=(f^t_{jk})^{-\frac{n-1}{2}}v(r, (f^t_{jk})^{-1}\theta)|\diff q|^{1/2}.  
\]

This is a unitary operator on $L^2(\mathbb{R}^n; \Omega^{1/2})$. Then the conjugation of $\psi_j\Op_\hbar (a)\psi_k: C_c^\infty(\mathbb{R}^n)\to C^\infty(\mathbb{R}^n)$ by $(\Theta^t_{jk})^*$ is 
\begin{align*}
&(\Theta^t_{jk})_*\psi_j\Op^t_\hbar (a)\psi_k(\Theta^t_{jk})^* (v|\diff q|^{1/2}) \\
&=\frac{1}{(2\pi\hbar)^n}\int_{\mathbb{R}^n}\diff q \int_{\mathbb{R}^n}\diff q^\prime \, a\left(r, (f^t_{jk})^{-1}\theta, p\right)\psi_j(r)\psi_k(r^\prime) \\
&\quad\times e^{i\rho(r-r^\prime)/\hbar+i\eta\cdot ((f^t_{jk})^{-1}\theta-\theta^\prime)/\hbar}v(r^\prime, f^t_{jk}\theta^\prime)|\diff q|^{1/2} \\
&=\frac{1}{(2\pi\hbar)^n}\int_{\mathbb{R}^n}\diff \tilde p \int_{\mathbb{R}^n}\diff \tilde{q}^\prime \, (({\widetilde \Theta}^t_{jk})_*a)\left(tq+(1-t)\tilde q^\prime, \rho, \tilde\eta\right) \\
&\quad \times \psi_j(r)\psi_k(r^\prime)e^{i\tilde p\cdot (q-\tilde{q}^\prime)/\hbar}v(\tilde{q}^\prime) |\diff q|^{1/2}\\
&=\psi_j\Op^t_\hbar (({\widetilde \Theta}^t_{jk})_*a)(\psi_k v(q) |\diff q|^{1/2}). 
\end{align*}
Here we changed the variables 
\[
\tilde{q}^\prime=(r, \tilde\theta)=(r, f^t_{jk}\theta), \quad \tilde p=(\rho, \tilde\eta)=(\rho, (f^t_{jk})^{-1}\eta). 
\]
Thus if $\psi_j\Op^t_\hbar (({\widetilde \Theta}^t_{jk})_*a)\psi_k$ is bounded on $L^2(\mathbb{R}^n; \Omega^{1/2})$, then $\psi_j\Op_\hbar (a)\psi_k$ is also bounded on $L^2(\mathbb{R}^n; \Omega^{1/2})$ and they have the same operator norm. 

\begin{lemm}\label{lemm_scaling_bisymbols}
    We define
    \[
        a^t_{jk}(q, p, q^\prime):=\psi_j (r)\psi_k (r^\prime)(({\widetilde \Theta}^t_{jk})_*a)(tq+(1-t)q^\prime, p)
    \]
    for $a \in \locsymb{m}{f}$. Then, for all multiindices $\alpha, \beta, \gamma \in \nni^n$, there exists a constant $C>0$ such that the estimate 
    \[
        |\partial_q^\alpha \partial_p^\beta \partial_{q^\prime}^\gamma a^t_{jk}(q, p, q^\prime)|\leq C\jbracket{p}^{m-\sigma|\beta|}
    \]
    holds for all $(q, p, q^\prime)\in \mathbb{R}^{3n}$, $j, k\in \mathbb{Z}$ and $a\in \locsymb{m}{f}$. 
\end{lemm}

\begin{proof}
We immediately obtain the result by differentiating both sides of 
\[
    a^t_{jk}(q, p, q^\prime)=\psi_j (r) \psi_k (r^\prime) a(tr+(1-t)r^\prime, (f^t_{jk})^{-1}(t\theta+(1-t)\theta^\prime), \rho, f^t_{jk}\eta),
\]
noting that there exists a positive constant $C>0$ such that $C^{-1}\leq f(tr+(1-t)r^\prime)/f^t_{jk}\leq C$ for all $j, k\in \mathbb{Z}$, $r\in \rmop{supp}\psi_j$ and $r^\prime \in \rmop{supp}\psi_k$ by the Assumption \ref{assu_f}. 
\end{proof}

\begin{prop}\label{prop_usual_bisymbols}
    For $m\in \mathbb{R}$ and $\sigma\in [0, 1]$, we define a class of smooth functions $S^0_\sigma (\mathbb{R}^{3n})$ as 
    \[
        S^m_\sigma (\mathbb{R}^{3n}):= \{ a(q, p, q^\prime)\in C^\infty (\mathbb{R}^{3n}) \mid |a|_{S^m_\sigma(\mathbb{R}^{3n}), N}<\infty, \, \forall N \in \nni\}, 
    \]
    where the seminorm $|a|_{S^m_\sigma(\mathbb{R}^{3n}), N}$ is defined as 
    \[
        |a|_{S^m_\sigma(\mathbb{R}^{3n}), N}:=\sum_{|\alpha|+|\beta|+|\gamma|\leq N}\| \jbracket{p}^{-m+\sigma |\beta|}\partial_q^\alpha \partial_p^\beta \partial_{q^\prime}^\gamma a\|_{L^\infty (\mathbb{R}^{3n})}. 
    \]
    For $a\in S^m_\sigma (\mathbb{R}^{3n})$, we define an operator $\Op_\hbar (a): C_c^\infty (\mathbb{R}^n; \Omega^{1/2})\to C^\infty (\mathbb{R}^n; \Omega^{1/2})$ as 
    \[
        \Op_\hbar (a)(v|\diff q|^{1/2}):=\frac{1}{(2\pi\hbar)^n}\int_{\mathbb{R}^{2n}} a(q, p, q^\prime)e^{ip\cdot (q-q^\prime)/\hbar}v(q^\prime)\, \diff q^\prime \diff p |\diff q|^{1/2}
    \]
    Then the following statements hold. 
\begin{enumerate}[label=(\roman*)]
\item \label{enum_bdd_usual}There exists a constant $C>0$ and $N\geq 0$ such that 
\[
\|\Op_\hbar (a)u\|_{L^2(\mathbb{R}^n; \Omega^{1/2})}\leq C|a|_{S^0_0(\mathbb{R}^{3n}), N}\|u\|_{L^2(\mathbb{R}^n; \Omega^{1/2})}
\]
holds for all $\hbar\in (0, 1]$, $a\in S^0_0(\mathbb{R}^{3n})$ and $u\in C_c^\infty(\mathbb{R}^n; \Omega^{1/2})$. 
\item \label{enum_disjoint_usual}For all $m\in \mathbb{R}$ and $N\in \nni$ and $\delta>0$, there exists a constant $C>0$ and an integer $N^\prime\geq 0$ such that 
\begin{align*}
&\|\chi_1\Op_\hbar (a)\chi_2u\|_{L^2(\mathbb{R}^n; \Omega^{1/2})} \\
&\leq C\hbar^N\delta^{-N}|\chi_1|_{N^\prime} |\chi_2|_{N^\prime}|a|_{S^m_1 (\mathbb{R}^{3n}), N+N^\prime} \|u\|_{L^2(\mathbb{R}^n; \Omega^{1/2})}
\end{align*}
holds for all $a\in S^m_1 (\mathbb{R}^{3n})$, $u\in C_c^\infty(\mathbb{R}^n)$ and $\chi_1$, $\chi_2\in \mathcal{B}$ with 
\[
\mathrm{dist}(\mathrm{supp}(\chi_1), \mathrm{supp}(\chi_2))\geq \delta. 
\]
Here 
\[
\mathcal{B}:=\{\, \chi\in C^\infty(\mathbb{R}^n)\mid \partial_q^\alpha\chi\in L^\infty(\mathbb{R}^n)\text{ for all } \alpha\in \nni^n\,\}
\]
and 
\[
|\chi|_{N^\prime}:=\sum_{|\alpha|\leq N^\prime} \|\partial_q^\alpha\chi\|_{L^\infty(\mathbb{R}^n)}. 
\]
\end{enumerate}
\end{prop}

\begin{proof}
    (i) This is a well-known result. See, for instance, the textbook of Kumano-go \cite{Kumano-go81} (Combine Theorem 2.5 in Chapter 2 with Theorem 1.6 in Chapter 7). 
    
    (ii) Since $q\neq q^\prime$ on the support of $a(q, p, q^\prime)\chi_1(q)\chi_2(q^\prime)$, we can apply the integration by parts by a differential operator $-i(q-q^\prime)\cdot\partial_p/|q-q^\prime|^2$ and obtain
    \[
    \chi_1\Op_\hbar (a)\chi_2u=\hbar^N \Op_\hbar \left(\chi_1(q)\chi_2(q^\prime)\left(\frac{i(q-q^\prime)\cdot\partial_p}{|q-q^\prime|^2}\right)^Na(q, p, q^\prime)\right)u 
    \]
    for all $N\in \nni$ and $u\in C_c^\infty(\mathbb{R}^n; \Omega^{1/2})$. 
    
    The derivative of the symbol is estimated as 
    \begin{align*}
    &\sum_{|\alpha|+|\beta|+|\gamma|\leq N^\prime}\left|\partial_q^\alpha\partial_p^\beta\partial_{q^\prime}^\gamma\left(\chi_1(q)\chi_2(q^\prime)\left(\frac{i(q-q^\prime)\cdot\partial_p}{|q-q^\prime|^2}\right)^Na(q, p, q^\prime)\right)\right| \\
    &\leq C_{NN^\prime}|\chi_1|_{N^\prime}|\chi_2|_{N^\prime}|a|_{S^m_1 (\mathbb{R}^{3n}), N+N^\prime}|q-q^\prime|^{-N}1_{\mathrm{supp}(\chi_1)}(q)1_{\mathrm{supp}(\chi_2)}(q^\prime) \\
    &\leq C_{NN^\prime}\delta^{-N}|\chi_1|_{N^\prime} |\chi_2|_{N^\prime} |a|_{S^m_1 (\mathbb{R}^{3n}), N+N^\prime}
    \end{align*}
    for all $N^\prime\in\nni$ and $N\geq m$. 
    
    Combining this estimate with Proposition \ref{prop_usual_bisymbols} \ref{enum_bdd_usual}, we finish the proof. 
    \end{proof}

\begin{proof}[Proof of Proposition \ref{prop_disjoint_union}]
By Proposition \ref{prop_usual_bisymbols} \ref{enum_bdd_usual} and Lemma \ref{lemm_scaling_bisymbols}, we obtain
\begin{align*}
&\|\psi_j\Op_\hbar (a)\psi_k\|_{L^2(\mathbb{R}^n; \Omega^{1/2})\to L^2(\mathbb{R}^n; \Omega^{1/2})} \\
&=\|\Op_\hbar (a^t_{jk})\|_{L^2(\mathbb{R}^n; \Omega^{1/2})\to L^2(\mathbb{R}^n; \Omega^{1/2})}\leq C|a^t_{jk}|_{S^0_0(\mathbb{R}^{3n}), N^\prime}\leq C|a|_{\locsymb{0}{f}, N^\prime}
\end{align*}
for some $C>0$ and $N^\prime\geq 0$ independent of $j$, $k\in \mathbb{Z}$. 

Assume that $|j-k|\geq 2$. Since $\mathrm{supp}(\psi)\subset (-1, 1)$, there exists a small $\delta>0$ such that $\mathrm{supp}(\psi)\subset (-1+\delta, 1-\delta)$. Take a smooth function $\tilde\psi: \mathbb{R}\to [0, \infty)$ such that
\[
\mathrm{supp}(\tilde\psi)\subset \left(-1+\frac{\delta}{2}, 1-\frac{\delta}{2}\right), \quad \tilde\psi=1 \text{ on } \mathrm{supp}(\psi) 
\]
and put $\tilde\psi_j:=\tilde\psi(\cdot-j)$. 
Then 
\[
\mathrm{dist}(\mathrm{supp}(\tilde\psi_j), \mathrm{supp}(\tilde\psi_k))\geq |j-k|-2+\delta>0
\]
and $\Op_\hbar (a^t_{jk})=\tilde\psi_j\Op_\hbar (a^t_{jk})\tilde\psi_k$. Hence we can apply Proposition \ref{prop_usual_bisymbols} \ref{enum_disjoint_usual} and obtain
\begin{align*}
&\|\Op_\hbar (a^t_{jk})u\|_{L^2(\mathbb{R}^n; \Omega^{1/2})}=\|\tilde\psi_j\Op_\hbar (a^t_{jk})\tilde\psi_ku\|_{L^2(\mathbb{R}^n; \Omega^{1/2})} \\
&\leq 
C_N(|j-k|-2+\delta)^{-N}|\tilde\psi|_{N^\prime}^2|a^t_{jk}|_{S^0_1(\mathbb{R}^{3n})}\|u\|_{L^2(\mathbb{R}^n; \Omega^{1/2})} \\
&\leq C_N\jbracket{j-k}^{-N}|a^t_{jk}|_{S^0_1 (\mathbb{R}^{3n}), N^\prime}\|u\|_{L^2(\mathbb{R}^n; \Omega^{1/2})} \\
&\leq C_N\jbracket{j-k}^{-N}|a|_{\locsymb{0}{f}, N^\prime}\|u\|_{L^2(\mathbb{R}^n; \Omega^{1/2})}
\end{align*}
for some $C>0$ and $N^\prime\geq 0$ independent of $a\in \locsymb{0}{f}$, $u\in C_c^\infty(\mathbb{R}^n; \Omega^{1/2})$ and $j, k\in \mathbb{Z}$. 
\end{proof}

We return to the proof of Theorem \ref{theo_L2_bdd_simplest}. We need two more lemmas. 
\begin{lemm}\label{lemm_formal_adjoint}
If $a\in \locsymb{m}f$, then
\[
    \jbracket{\psi_j\Op^t_\hbar (a)\psi_k u, v}
    =\jbracket{u, \psi_k\Op^{1-t}(\overline{a})\psi_j u} 
\]
for all $a\in \locsymb{m}{f}$ and $u, v\in C_c^\infty (\mathbb{R}^n; \Omega^{1/2})$. 
\end{lemm}

\begin{proof}
Proposition \ref{prop_smoothness} justifies the changing of order of integration. 
\end{proof}

\begin{lemm}[Cotlar-Stein lemma]\label{lemm_cotlar_stein}
Let $\{A_\alpha: \mathcal{H}_1\to \mathcal{H}_2\}_{\alpha\in \Lambda}$ be a countable family of bounded operators between two Hilbert spaces $\mathcal{H}_1$ and $\mathcal{H}_2$. If 
\[
\sup_{\alpha\in \Lambda}\sum_{\beta\in \Lambda}\|A_\alpha^*A_\beta\|^{1/2}\leq M \text{ and } 
\sup_{\alpha\in \Lambda}\sum_{\beta\in \Lambda}\|A_\alpha A_\beta^*\|^{1/2}\leq M, 
\]
then 
\[
A:=\sum_{\alpha\in \Lambda}A_\alpha
\]
converges in a strong operator topology and $\|A\|\leq M$. 
\end{lemm}

We can find the proof of the Cotlar-Stein lemma in \cite{Martinez02}, \cite{Zworski12} for instance. 

\begin{proof}[Proof of Theorem \ref{theo_L2_bdd_simplest}]
We set $A_{jk}:=\psi_j\Op^t_\hbar (a)\psi_k$. By the Calder\'on-Vaillancourt theorem, $A_{jk}$ are bounded operators on $L^2(\mathbb{R}^n; \Omega^{1/2})$. Thus the adjoint operator $A_{jk}^*$ exists. Since 
\[
A_{jk}A_{lm}^*=0 \quad \text{if } |k-m|\geq 2
\]
and 
\[
A_{jk}^*A_{lm}=0 \quad \text{if } |j-l|\geq 2
\]
by Lemma \ref{lemm_formal_adjoint}, we obtain the estimates
\[
\sum_{l, m\in \mathbb{Z}}\|A_{jk}A_{lm}^*\|^{1/2}\leq C|a|_{\locsymb{0}{f}, N}\sum_{\substack{l, m\in \mathbb{Z} \\ |k-m|\leq 1}}\jbracket{j-k}^{-3/2}\jbracket{l-m}^{-3/2}\leq C|a|_{\locsymb{0}{f}, N}
\]
and 
\[
\sum_{l, m\in \mathbb{Z}}\|A_{jk}^*A_{lm}\|^{1/2}\leq C|a|_{\locsymb{0}{f}, N}\sum_{\substack{l, m\in \mathbb{Z} \\ |j-l|\leq 1}}\jbracket{j-k}^{-3/2}\jbracket{l-m}^{-3/2}\leq C|a|_{\locsymb{0}{f}, N}
\]
for some $C>0$ and $N\geq 0$ independent of $t\in [0, 1]$, $a\in \locsymb{0}{f}$ and $j$, $k\in\mathbb{Z}$ by Proposition \ref{prop_disjoint_union}. Hence, by the Cotlar-Stein lemma (Lemma \ref{lemm_cotlar_stein}), we obtain
\[
\left\|\sum_{j, k\in \mathbb{Z}}A_{jk}\right\|_{L^2(\mathbb{R}^n; \Omega^{1/2})\to L^2(\mathbb{R}^n; \Omega^{1/2})}\leq C|a|_{\locsymb{0}{f}, N}. \qedhere
\]
\end{proof}

\subsection{Changing angular coordinates}

For constructing pseudodifferential operators on manifolds, we need to change angular variables. The following theorem will be employed in the proof of Theorem \ref{theo_psido_local_global}. 

\begin{theo}\label{theo_changing_ang_variables}
    Let $\varphi^\prime: V^\prime_1 \to V^\prime_2 $ be a diffeomorphism between bounded open subsets $V^\prime_1, V^\prime_2$ of $\mathbb{R}^{n-1}$ and set $\varphi:=\rmop{id}\times \varphi^\prime: \mathbb{R}\times V^\prime_1 \to \mathbb{R}\times V^\prime_2$. Then for all $a\in \locsymb{m}{f}$ with $\rmop{supp}a \subset \mathbb{R}\times V_1^\prime$, there exists a symbol $a_\varphi\in \locsymb{m}{f}$ such that the relation 
    \[
        \varphi_*\Op^t_\hbar  (a)\varphi^*=\Op^t_\hbar  (a_\varphi)
    \]
    holds. Moreover, $a_\varphi$ has an asymptotic expansion 
    \[
        a^t_\varphi 
        \sim \sum_{j=0}^\infty \hbar^j a^t_{\varphi, j}, 
        \quad a_{\varphi, j}\in \locsymb{m- j}{f}
    \]
    with 
    \[
        a^t_{\varphi, 0}(q, p)=\tilde\varphi_* a(q, p). 
    \]
    and $\rmop{supp}a^t_{\varphi, j}\subset \rmop{supp}\tilde\varphi_* a$ for all $j \in \nni$. 
\end{theo}

\begin{proof}
    By a direct calculation, we have 
    \begin{align*}
        &\varphi^*\Op^t_\hbar  (a)\varphi_*(v|\diff q|^{1/2}) \\
        &=\frac{1}{(2\pi\hbar)^n}\int_{\mathbb{R}^{2n}}b^t_\varphi (q, p, q^\prime) e^{ip\cdot (\varphi(q)-\varphi(q^\prime))/\hbar}v(q^\prime)\, \diff q^\prime \diff p^\prime |\diff q|^{1/2}
    \end{align*}
    where 
    \[
        b^t_\varphi (q, p, q^\prime):=a(t\varphi (q)+(1-t)\varphi (q^\prime), p)|\det \partial \varphi^\prime (\theta)|^{1/2}|\det \partial \varphi^\prime (\theta^\prime)|^{1/2}. 
    \]
    We note that 
    \[
        t\varphi (q)+(1-t)\varphi (q^\prime)=(tr+(1-t)r, t\varphi^\prime (\theta)+(1-t)\varphi^\prime (\theta^\prime))
    \]
    and 
    \[
        p\cdot (\varphi (q)-\varphi (q^\prime))=\rho (r-r^\prime)+\eta\cdot (\varphi^\prime (\theta)-\varphi^\prime (\theta^\prime)). 
    \]
    Then we find a symbol $a^t_\varphi (q, p)$ satisfying  
    \[
        \int_{\mathbb{R}^{n-1}}b^t_\varphi (q, p, q^\prime) e^{i\eta\cdot (\varphi^\prime (\theta)-\varphi(\theta^\prime))/\hbar}\, \diff \eta
        =\int_{\mathbb{R}^{n-1}}a_\varphi (tq+(1-t)q^\prime, p) e^{i\eta\cdot (\theta-\theta^\prime)/\hbar}\, \diff \eta
    \]
    and obtain 
    \begin{equation}
        \label{eq_btphi}
        \begin{aligned}
        &a^t_\varphi (q, p) \\
        &=\frac{1}{(2\pi\hbar)^{n-1}} \int_{\mathbb{R}^{2n-2}} 
        b(q, p, \theta^\prime, \eta^\prime)
        e^{i\eta^\prime \cdot (\varphi^\prime (\theta+(1-t)\theta^\prime)-\varphi^\prime (\theta-t\theta^\prime))/\hbar-i\eta\cdot \theta^\prime/\hbar}\, \diff \theta^\prime \diff \eta^\prime
        \end{aligned}
    \end{equation}
    where 
    \begin{align*}
        b(q, p, \theta^\prime, \eta^\prime) 
        :=&a(r, t\varphi^\prime (\theta+(1-t)\theta^\prime)+(1-t)\varphi^\prime (\theta-t\theta^\prime), \rho, \eta^\prime) \\
        &\times |\det \partial \varphi^\prime (\theta+(1-t)\theta^\prime)|^{1/2}|\det \partial \varphi^\prime (\theta-t\theta^\prime)|^{1/2}. 
    \end{align*}
    We define $\Theta_j(r, \theta):=(r, f(j)\theta)$ and ${\widetilde\Theta}_j(r, \theta, \rho, \eta):=(r, f(j)\theta, \rho, f(j)^{-1}\eta)$. Then we have 
    \[
        |\partial_q^\alpha \partial_p^\beta (\widetilde{\Theta}_{j*}a_\varphi)(q, p)|
        \leq C\jbracket{\rho \oplus \frac{f(j)}{f(r)}\eta}^{m-|\beta|}
        \leq C\jbracket{p}^{m-|\beta|}
    \]
    for all $|r-j|\leq 1$ and a constant $C>0$ independent of $j$. This implies that $a_\varphi \in \locsymb{m}{f}$. 

    Moreover, we have 
    \[
        \widetilde{\Theta}_{j*}a^t_\varphi (q, p)\sim \sum_{k=0}^\infty \frac{\hbar^k}{k!}\left(-\frac{i}{2}\right)^k (D_{\theta^\prime}\cdot D_{\eta^\prime})^k b(q, p, 0, 0)
    \]
    uniformly in $r\in [j-1, j+1]$ and $j\in \mathbb{Z}$ by the method of stationary phase. 
\end{proof}

\begin{lemm}
    \label{lemm_multiplication}
    Let $a(q, p)\in \locsymb{m}{f}$ and $\chi (q)\in \locsymb{0}{1}$ (not in general $\locsymb{0}{f, 1}$). Then there exist symbols $\chi \#^t a, a\#^t \chi \in \locsymb{m}{f}$ such that 
    \begin{align}
        &\chi \Op^t_\hbar  (a)=\Op^t_\hbar  (\chi \#^t a), \label{eq_multiplication_left}\\
        &\Op^t_\hbar  (a)\chi=\Op^t_\hbar  (a \#^t \chi). \label{eq_multiplication_right}
    \end{align}
    Moreover, $\chi \#^t a$ and $a\#^t \chi$ have asymptotic expansions 
    \begin{equation}\label{eq_asymptotic_left}
        (\chi\#^t a)(q, p)\sim \sum_{j=0}^\infty 
        \frac{(i\hbar (1-t))^j}{j!}(\partial_{q^\prime}\cdot \partial_p)^j (a(q, p)\chi (q^\prime))|_{q^\prime=q}
    \end{equation}
    and 
    \begin{equation}
        \label{eq_asymptotic_right}
        (a\#^t \chi)(q, p)\sim \sum_{j=0}^\infty 
        \frac{(i\hbar t)^j}{j!}(\partial_{q^\prime}\cdot \partial_p)^j (a(q, p)\chi (q^\prime))|_{q^\prime=q}. 
    \end{equation}
\end{lemm}

\begin{proof}
    Since $\chi (q)a(tq+(1-t)q^\prime, p)\in B\locsymb{m}{f, t}$, we can apply Proposition \ref{prop_bisymbol_symbol} and obtain a symbol $\chi \#^t a\in \locsymb{m}{f}$ satisfying \eqref{eq_multiplication_left} and having an asymptotic expansion 
    \[
        (\chi \#^t a)(q, p)
        \sim 
        \sum_{j=0}^\infty \frac{\hbar^j}{j !}(i\partial_p\cdot \partial_{q^\prime})^j|_{q^\prime=0}b^t (q, p, q^\prime)
    \]
    where 
    \[
        b^t (q, p, q^\prime):=\chi(q+(1-t)q^\prime)a(t(q+(1-t)q^\prime)+(1-t)(q-tq^\prime), p). 
    \]
    This asymptotic expansion is equivalent to the desired expansion \eqref{eq_asymptotic_left}. 

    We can prove \eqref{eq_multiplication_right} and \eqref{eq_asymptotic_right} similarly. 
\end{proof}

\subsection{Pseudodifferential operators on manifolds}\label{subs_psido_manifolds}

Before defining pseudodifferential operators, we introduce a class of functions dependent only on angular variables near infinity. 

\begin{defi}
    We call a function $u: M\to \mathbb{C}$ is cylindrical if there exist $R>0$ and a function $u_\mathrm{ang}: S\to \mathbb{C}$ such that $\Psi_* u (r, \theta)=u_\mathrm{ang}(\theta)$ for all $(r, \theta)\in [R, \infty)\times S$. 
\end{defi}

We introduce a partition of unity $\{\kappa_\iota\}_{\iota\in I}$ consisting of cylindrical functions subordinated to the open covering $\{U_\iota\}_{\iota\in I}$. We take a collection of subsets $\{\Gamma_\iota\}_{\iota\in I}$ in Definition \ref{defi_symbol_class_manifold} such that $\rmop{supp}\kappa_\iota \subset \varphi_\iota^{-1}(\Gamma_\iota)$ for all $\iota\in I$. 

\begin{defi}
    Take a collection of cylindrical functions $\{\chi_\iota\}_{\iota\in I}$ such that $\rmop{supp}\chi_\iota \subset U_\iota$ and $\chi_\iota=1$ near $\rmop{supp}\kappa_\iota$. For $a\in S^m_f (T^*M)$, we define a pseudodifferential operator $\Op^t_{M, \hbar}(a)$ with the symbol $a$ as 
    \begin{equation}\label{eq_psido_manifold_defi}
     \Op^t_{M, \hbar}(a)   
    :=\sum_{\iota\in I} \chi_\iota \varphi_\iota^* \Op^t_\hbar  (\tilde\varphi_{\iota*}(\kappa_\iota a))\varphi_{\iota*}\chi_\iota.  
    \end{equation}
\end{defi}

We can associate a pseudodifferential operator from a collection of locally defined symbols. 

\begin{theo}
    \label{theo_psido_local_global}
    Assume that $a_\iota\in \locsymb{m}{f}$ satisfies $\rmop{supp}a_\iota \subset \rmop{supp}\varphi_{\iota*}\kappa_\iota$. Then there exists a symbol $a\in S^m_f (T^*M)$ such that 
    \[
        \sum_{\iota\in I} \chi_\iota \varphi_\iota^* \Op^t_\hbar  (a_\iota)\varphi_{\iota*}\chi_\iota 
        =\Op^t_{M, \hbar}(a)+O_{L^2\to L^2}(\hbar^\infty). 
    \]
    The symbol $a$ has an asymptotic expansion 
    \[
        a(q, p)\sim \sum_{j=0}^\infty \hbar^j a_j (q, p), \quad a_j \in S^{m-j}_f (T^*M)
    \]
    with 
    \[
        a_0(q, p):=\sum_{\iota\in I} \tilde\varphi_\iota^*a_\iota (q, p). 
    \]
\end{theo}

\begin{proof}
    We calculate $\Op^t_{M, \hbar} (a_0)$ explicitly. Then we have 
    \begin{align}
        \Op^t_{M, \hbar} (a_0)
        &=\sum_{U_\iota \cap U_{\iota^\prime}\neq \varnothing} 
        \chi_\iota \varphi_\iota^* \Op^t_\hbar (\tilde\varphi_{\iota*}(\kappa_\iota \tilde\varphi_{\iota^\prime}^*a_{\iota^\prime}))(\varphi_{\iota*}\chi_\iota)\varphi_{\iota*} \nonumber\\
        &=\sum_{U_\iota \cap U_{\iota^\prime}\neq \varnothing} 
        \chi_\iota \varphi_{\iota^\prime}^* \Op^t_\hbar (a_{\iota^\prime}\varphi_{\iota^\prime *}\kappa_\iota+O_{\locsymb{m-1}{f}}(\hbar))(\varphi_{\iota^\prime *}\chi_\iota)\varphi_{\iota^\prime *} \nonumber\\
        &=\sum_{\iota^\prime \in I} 
        \varphi_{\iota^\prime}^* \Op^t_\hbar (a_{\iota^\prime}-\hbar b_{1, \iota^\prime})\varphi_{\iota^\prime *}+O_{L^2\to L^2}(\hbar^\infty)  \label{eq_a0_local_global_III}
    \end{align}
    by Theorem \ref{theo_changing_ang_variables}, the assumption $\rmop{supp}a_\iota \subset \rmop{supp}\varphi_{\iota*}\kappa_\iota$ and Lemma \ref{lemm_multiplication}. Here $b_{1, \iota}\in \locsymb{m-1}{f}$ has an asymptotic expansion 
    \[
        b_{1, \iota}(x, \xi)\sim \sum_{j=0}^\infty \hbar^j b_{1j, \iota}(x, \xi), \quad b_{1j, \iota}\in \locsymb{m-j-1}{f}
    \]
    with $\rmop{supp} b_{1j, \iota}\subset \rmop{supp}\varphi_{\iota*}(\kappa_\iota a_0)$. 

    We repeat the same argument for $b_{10, \iota}(x, \xi)$. If we set 
    \[
        a_1(x, \xi):=-\sum_{\iota\in I} \tilde\varphi_\iota^* b_{10, \iota}(x, \xi), 
    \]
    then we have 
    \begin{equation}\label{eq_a1_local_global_III}
        \Op^t_{M, \hbar} (a_1)=\sum_{\iota\in I} \Op^t_\hbar \left( b_{10, \iota}-\hbar c_{2, \iota}\right)+O_{L^2\to L^2}(\hbar^\infty). 
    \end{equation}
    Here $c_{2, \iota}(x, \xi)\in \locsymb{m-2}{f}$ has an asymptotic expansion 
    \[
        c_{2, \iota}(x, \xi)\sim \sum_{j=0}^\infty c_{2j, \iota}(x, \xi), \quad c_{2j, \iota}\in \locsymb{m-j-2}{f}
    \]
    with $\rmop{supp} c_{2j, \iota}\subset \rmop{supp}\varphi_{\iota*}(\kappa_\iota a_0)$. 

    Summing up \eqref{eq_a0_local_global_III} and \eqref{eq_a1_local_global_III}, we obtain 
    \[
        \Op^t_{M, \hbar} (a_0+\hbar a_1)=\sum_{\iota\in I} \Op^t_\hbar (a_\iota-\hbar^2 b_{2, \iota})+O_{L^2\to L^2}(\hbar^\infty), 
    \]
    where 
    \[
        b_{2, \iota}(x, \xi):=\hbar^{-1}(b_{1, \iota}-b_{10, \iota})+c_{2, \iota} \in \locsymb{m-2}{f}. 
    \]
    $b_{2, \iota}(x, \xi)$ has an asymptotic expansion 
    \[
        b_{2, \iota}(x, \xi)\sim \sum_{j=0}^\infty \hbar^j b_{2j, \iota}(x, \xi), \quad b_{2j, \iota}\in \locsymb{m-j-2}{f}
    \]
    with $\rmop{supp} b_{2j, \iota}\subset \rmop{supp}\varphi_{\iota*}(\kappa_\iota a_0)$. 

    We repeat this argument and construct $a_j\in S^{m-j}_f (T^*M)$ such that 
    \[
        \Op^t_{M, \hbar} \left( \sum_{j=0}^N \hbar^j a_j\right)
        =\sum_{\iota\in I} \Op^t_\hbar (a_\iota- \hbar^{N+1}b_{N+1, \iota})+O_{L^2\to L^2}(\hbar^\infty)
    \]
    for all $N\in \mathbb{Z}_{\geq 0}$, where $b_{N+1}(x, \xi)\in S^{N+1-j}(T^*\mathbb{R}^n)$ has an asymptotic expansion
    \[
        b_{N+1, \iota}(x, \xi)\sim \sum_{j=0}^\infty \hbar^j b_{N+1, j, \iota}(x, \xi), \quad b_{N+1, j, \iota}\in \locsymb{m-j-N-1}{f}. 
    \]

    The desired symbol $a(x, \xi)$ is defined as an asymptotic expansion 
    \[
        a(x, \xi)\sim \sum_{j=0}^\infty \hbar^j a_j(x, \xi)
    \]
    by the Borel theorem. 
\end{proof}

Next we prove the boundedness on $L^2$ space of pseudodifferential operators with symbols in $S^0_f(T^*M)$. 

\begin{theo}\label{theo_L2_bdd_manifold}
    There exists a constant $C>0$ and an integer $N\in \nni$ such that the estimate 
    \[
        \|\Op^t_{M, \hbar}(a)\|_{L^2\to L^2}\leq C|a|_{S^0_f (T^*M), N}
    \]    
    holds for all $a\in S^0_f(T^*M)$, $t\in [0, 1]$ and $\hbar\in (0, 1]$. 
\end{theo}

\begin{proof}
    Each term $\chi_\iota \varphi_\iota^* \Op^t_\hbar  (\tilde\varphi_{\iota*}(\kappa_\iota a))\varphi_{\iota*}\chi_\iota$ in the definition \eqref{eq_psido_manifold_defi} of pseudodifferential operators is bounded on $L^2(M; \Omega^{1/2})$ by Theorem \ref{theo_L2_bdd_simplest} and the unitarity of the pullback $\varphi^*_\iota$ and pushforward $\varphi_{\iota*}$. Thus we obtain the boundedness of $\Op^t_{M, \hbar}(a)$. 
\end{proof}

\section{Resolvents of differential operators}\label{sec_ess_selfadj}

\subsection{Differential operators on manifolds}

We define a class of differential operators on manifolds. For the definition, we only employ the function $f: \mathbb{R}\to (0, \infty)$ appeared in Assumption \ref{assu_f}, and thus the definition is independent of Riemannian metrics on $M$. 

\begin{defi}\label{def_elliptic_M}
An operator $P_\hbar : C^\infty(M; \Omega^{1/2})\to C^\infty(M; \Omega^{1/2})$ is a \textit{semiclassical differential operator on} $M$ \textit{of degree at most} $m$ if $P_\hbar$ satisfies the following conditions.  
\begin{itemize}
\item $\mathrm{supp}(P_\hbar u)\subset \mathrm{supp}(u)$ for all $u\in C^\infty(M)$.  
\item For all $\iota\in I_\infty$, $\varphi_{\iota*} P_\hbar  \varphi_\iota^*$ is represented as 
\begin{equation}\label{eq_diff_local_polar}
    \varphi_{\iota*} P_\hbar  \varphi_\iota^*(v|\diff r \diff \theta|^{1/2})
    =\sum_{|\alpha|\leq m} p_\alpha (\hbar; r, \theta) (f(r)^{-1}\hbar D_\theta)^{\alpha^\prime} (\hbar D_r)^{\alpha_0}
    v |\diff r \diff \theta|^{1/2}
\end{equation}
with coefficients $p_\alpha (\hbar; r, \theta)\in C^\infty (\mathbb{R}^n)$ such that 
\[
    p_\alpha (\hbar; r, \theta)=\sum_{j=0}^{N_\alpha}\hbar^j p_{\alpha, j} (r, \theta)
\]
for some $N_\alpha\in \nni$ and 
\[
    \|(f(r)^{-1}\partial_\theta)^{\alpha^\prime} \partial_r^{\alpha_0}p_{\alpha, j} (r, \theta)\|_{L^\infty ([R, \infty)\times K)}
    \leq C_{KR\alpha} 
\]
for all $R>0$ and all compact subsets $K\subset \mathbb{R}^{n-1}$. 
\item For all $\iota\in I_K$, $\varphi_{\iota*} P_\hbar  \varphi_\iota^*$ is represented as 
\begin{equation}\label{eq_diff_local_comp}
    \varphi_{\iota*} P_\hbar  \varphi_\iota^*(v|\diff x|^{1/2})
    =\sum_{|\alpha|\leq m} p_\alpha (\hbar; x) (\hbar D_x)^\alpha
    v |\diff x|^{1/2}
\end{equation}
with coefficients $p_\alpha (\hbar; x)\in C^\infty (\mathbb{R}^n)$ such that 
\[
    p_\alpha (\hbar; x)=\sum_{j=0}^{N_\alpha}\hbar^j p_{\alpha, j} (x)
\]
for some $N_\alpha \in \nni$ and 
\[
    \|\partial_x^\alpha p_{\alpha, j} (x)\|_{L^\infty (K)}
    \leq C_{K\alpha} 
\]
for all compact subsets $K\subset \mathbb{R}^n$. 
\end{itemize}
The set of all semiclassical differential operators on $ M $ of degree at most $m$ is denoted as $\mathrm{Diff}^m_{f, \hbar} (M; \Omega^{1/2})$. 
The principal symbol $\sigma(P_\hbar): T^*M \to \mathbb{C}$ of $P \in \mathrm{Diff}^m_{f, \hbar} (M; \Omega^{1/2})$ is defined as 
\[
\sigma(P_\hbar)(\tilde\varphi^{-1}(r, \theta, \rho, \eta)):=\sum_{|\alpha|\leq m} p_{\alpha, 0} (r, \theta) (f(r)^{-1}\eta)^{\alpha^\prime} \rho^{\alpha_0}
\]
in the notation of \eqref{eq_diff_local_polar} for $\iota\in I_\infty$ and 
\[
    \sigma(P_\hbar)(\tilde\varphi^{-1}(x, \xi)):=\sum_{|\alpha|\leq m} p_{\alpha, 0} (x) \xi^\alpha
\]
in the notation of \eqref{eq_diff_local_comp} for $\iota\in I_K$. 
$\sigma( P_\hbar )$ is independent of the choice of such $\iota\in I$. $\tilde \varphi_\iota: T^* U_\iota\to  V_\iota\times \mathbb{R}^n$ is the canonical coordinates associated with $ \varphi_\iota:  U_\iota\to V_\iota$. 
\end{defi}

We define the ellipticity of differential operators on $M$. 

\begin{defi}
A differential operator $ P_\hbar \in \mathrm{Diff}^m_{f, \hbar} (M; \Omega^{1/2})$ with $\sigma (P_\hbar)(T^*M)\neq \mathbb{C}$ is \textit{elliptic} if, for all $z\in\mathbb{C}$ with $\mathrm{dist}(z, \sigma( P_\hbar )(T^* M ))>0$, the following conditions hold. 
\begin{itemize}
    \item For any $\iota \in I_\infty$, there exists a constant $C>0$ such that the inequality
    \[
    C^{-1}\jbracket{\rho\oplus f(r)^{-1}\eta}^m\leq |z-\tilde\varphi_{\iota*}\sigma(P_\hbar)(q, p)| \leq C\jbracket{\rho\oplus f(r)^{-1}\eta}^m
    \]
    holds for all $(x, \xi)\in V_\iota \times \mathbb{R}^n$. 
    \item For any $\iota \in I_K$, there exists a constant $C>0$ such that the inequality
    \[
    C^{-1}\jbracket{\xi}^m\leq |z-\tilde\varphi_{\iota*}\sigma(P_\hbar)(x, \xi)| \leq C\jbracket{\xi}^m
    \]
    holds for all $(x, \xi)\in V_\iota \times \mathbb{R}^n$. 
\end{itemize}
\end{defi}

We give examples of differential operators in our sense. First example is the Lie derivatives acting on half-densities. 

\begin{exam*}[Lie derivatives]
Let $X$ be a vector field on $M$. The Lie derivative acting on half-densities is defined as 
\begin{equation}\label{eq_defi_lie}
    \mathcal{L}_X u:=\left.\diffrac[]{t}\right|_{t=0} \varphi^*_t u
\end{equation}
where $\varphi_t(x)$ is the flow generated by $X$. In general coordinates $(x_1, \ldots, x_n)$, the Lie derivative is calculated as 
\[
    \mathcal{L}_X (v |\diff x|^{1/2})=\left(Xv+\frac{1}{2}v \sum_{j=1}^n \frac{\partial X_j}{\partial x_j}\right) |\diff x|^{1/2}
\] 
where $X=\sum_j X_j \partial_{x_j}$. Assume that the vector field $X$ such that 
\[
    X=X_1(q)\frac{\partial}{\partial r}+f(r)^{-1}\sum_{j=1}^{n-1} Y_j (q)\frac{\partial}{\partial \theta_j} 
\] 
with $X_1\in \mathcal{B}_f(V_\iota)$ and $Y_j\in \mathcal{B}_f(V_\iota)$ in any polar coordinate function $\varphi: U\to V$ in the sense of \eqref{eq_polar_defi_III}. 
For a polar coordinate function $\varphi=(r, \theta)$, we have 
\[
    \varphi_*\mathcal{L}_X \varphi^*(v|\diff q|^{1/2})
    =\left(Xv+\frac{1}{2}v \left(\frac{\partial X_1}{\partial r}+f(r)^{-1}\sum_{j=1}^{n-1}\frac{\partial Y_j}{\partial \theta_j}\right)\right) |\diff q|^{1/2}. 
\]
This equation shows that $i^{-1}\hbar \mathcal{L}_X$ belongs to $\mathrm{Diff}^1_{f, \hbar} (M; \Omega^{1/2})$ and the principal symbol is 
\[
  \sigma (i^{-1}\hbar\mathcal{L}_X)(q, p)
  =\jbracket{p, X(q)}
  =X_1(q)\rho+f(r)^{-1}\sum_{j=1}^{n-1} Y_j(q)\eta_j.   
\]

We note that the Lie derivative is calculated as 
\[
    \mathcal{L}_X (v|\vol_g|^{1/2})=\left(Xv+\frac{1}{2}v \rmop{div} X\right) |\vol_g|^{1/2}
\]
on any Riemannian manifold $(M, g)$. 
\end{exam*}

Second example is the Laplacian acting on half-densities. Before we give the example concretely, we assume a compatibility condition for the structure of ends and the Riemannian metric. 

\begin{assu}\label{assu_metric_M}
    The pushforward of $g$ by $\Psi$ is the form
    \[
        \Psi_* g(r, \theta, \diff r, \diff \theta)=h(r, \theta, \diff r, f(r)\diff \theta) 
    \]
    where $h(r, \theta, \diff r, \diff \theta)$ satisfies the following. 
    \begin{itemize}
    \item There exists a constant $C>0$ such that the inequality 
    \[
        C^{-1}h(1, \theta, \diff r, \diff \theta)\leq h(r, \theta, \diff r, \diff \theta) \leq Ch(1, \theta, \diff r, \diff \theta) 
    \]
    holds for all $(r, \theta)\in [1, \infty)\times S$. 
    \item If $(\rmop{id}\times \varphi^\prime_\iota)_*h=\sum_{j, k=1}^n h^\iota_{ij}(r, \theta) \diff q_j \diff q_k$ ($q_1=r$ and $q_j=\theta_{j-1}$ for $j\geq 1$), then the coefficients $h^\iota_{jk} (r, \theta)$ belong to $\mathcal{B}_f (\mathbb{R}_+\times V_\iota^\prime)$. 
    \end{itemize}
\end{assu}

    The Assumption \ref{assu_metric_M} includes not only the case of manifolds with conical ends $f(r)=r$ (for instance, the Euclidean space), but also that of manifolds with asymptotically hyperbolic ends $f(r)=e^r$ (for instance, the equivalent setting in Mazzeo-Melrose \cite{Mazzeo-Melrose87}). 

    Now we show that the Laplacian gives an example of differential operators in our sense. 

\begin{exam*}[Laplacian]
    For a Riemannian manifold $(M, g)$, the Laplacian $\triangle_g$ acting on half-densities is defined as 
    \[
        \triangle_g (v |\vol_g|^{1/2}):=g^{-1/2}\sum_{j, k=1}^n \frac{\partial}{\partial x_j}\left(g^{1/2}g^{jk}\frac{\partial v}{\partial x_k}\right) |\vol_g|^{1/2}. 
    \]
    Then, for a general coordinate function $\varphi$, we have  
    \[
        \varphi_*\triangle_g\varphi^*(v|\diff x|^{1/2})
        =\left( \sum_{j, k=1}^n \frac{\partial}{\partial x_j}\left(g^{jk}\frac{\partial v}{\partial x_k}\right)+V_\varphi  v\right) |\diff x|^{1/2}
    \]
    where 
    \[
        V_\varphi :=-\frac{1}{2}g^{-1/4}\sum_{j, k=1}^n \frac{\partial}{\partial x_j}\left(g^{-1/4}g^{jk}\frac{\partial g^{1/2}}{\partial x_k} \right). 
    \]
    Hence, by Assumption \ref{assu_metric_M}, we have $V_\varphi \in \mathcal{B}_f$. Thus $-\hbar^2 \triangle_g \in \rmop{Diff}^2_f (M; \Omega^{1/2})$ and 
    \[
        \sigma (-\hbar^2 \triangle_g)(x, \xi)=|\xi|_{g^*}^2. 
    \]
    In particular, if the metric $g$ is the form $g=\diff r^2+f(r)^2 h(\theta, \diff \theta)$, then 
    \[
        V_\varphi (r, \theta)=-\frac{(n-1)^2}{4}\left(\diffrac[]{r} \log f(r)\right)^2-\frac{n-1}{2}\frac{\diff^2}{\diff r^2} \log f(r)
        -f(r)^{-2}V^\prime_\varphi (\theta)
    \]
    where 
    \[
        V^\prime_\varphi (\theta):=-\frac{1}{2}h^{-1/4}\sum_{j, k=1}^{n-1}\frac{\partial}{\partial \theta_j}\left(h^{-1/4}h^{jk}\frac{\partial h^{1/2}}{\partial \theta_k} \right). 
    \]
\end{exam*}

\subsection{Composition of differential operators and pseudodifferential operators}

In this section, we consider quantization procedures $\Op^1_\hbar$. 

\begin{theo}
    \label{theo_diff_psido_composition}
    Let $P\in \mathrm{Diff}^{m_1}_{f, \hbar} (M; \Omega^{1/2})$ and $a\in S^{m_2}_f (T^*M)$. Then there exists a symbol $b\in S^{m_1+m_2}_f (T^*M)$ such that 
    \[
        P_\hbar\Op^1_{M, \hbar} (a)=\Op^1_{M, \hbar} (b)+O_{L^2\to L^2}(\hbar^\infty). 
    \]
    $b(x, \xi)$ has an asymptotic expansion 
    \begin{equation}\label{eq_asymptotic_b_III}
        b(x, \xi)\sim \sum_{j=0}^\infty \hbar^j b_j (x, \xi), \quad b_j\in S^{m_1+m_2-j}_f (T^*M)
    \end{equation}
    and 
    \begin{equation}\label{eq_principal_b}
        b_0(x, \xi)=\sigma (P_\hbar)(x, \xi)a(x, \xi). 
    \end{equation}
\end{theo}

\begin{proof}
    We decompose $P_\hbar\Op^1_{M, \hbar} (a)$ into 
    \[
        P_\hbar \Op^1_{M, \hbar} (a)=\sum_{\iota\in I} P_\hbar\chi_\iota\varphi_\iota^* \Op^1_\hbar (\tilde\varphi_{\iota*}(\kappa_\iota a))\varphi_{\iota*}\chi_\iota. 
    \]
    We consider the case of $\iota\in I_\infty$. In polar coordinates, $P$ is written as 
    \[
        \varphi_{\iota*}P_\hbar \varphi_\iota^* =\sum_{|\alpha|\leq m} p^\iota_\alpha (\hbar; q)(f(r)^{-1}\hbar D_\theta)^{\alpha^\prime}(\hbar D_r)^{\alpha_0}
    \]
    with $p^\iota_\alpha \in \mathcal{B}_f$. Then 
    \begin{equation}
        \label{eq_composition_local}
        \begin{aligned}
            &P_\hbar\chi_\iota\varphi_\iota^* \Op^1_\hbar (\tilde\varphi_{\iota*}(\kappa_\iota a))\varphi_{\iota*}\chi_\iota \\
        &=\sum_{|\alpha|\leq m} \varphi_\iota^* p^\iota_\alpha (\hbar; q)(f(r)^{-1}\hbar D_\theta)^{\alpha^\prime}(\hbar D_r)^{\alpha_0}\Op^1_\hbar (\tilde\varphi_{\iota*}(\kappa_\iota a))\varphi_{\iota*}\chi_\iota \\
        &=\chi_\iota\varphi_\iota^* \Op^1_\hbar (b_\iota)\varphi_{\iota*}\chi_\iota. 
        \end{aligned}
    \end{equation}
    Here 
    \begin{equation}
        \label{eq_b_iota}
        b_\iota(q, p):=\sum_{|\alpha|\leq m} p^\iota_\alpha (\hbar; q)f(r)^{-|\alpha^\prime|}(p+\hbar D_q)^\alpha \tilde\varphi_{\iota*}(\kappa_\iota a)(q, p). 
    \end{equation}
    A direct calculation shows that $b_\iota$ belongs to $\locsymb{m_1+m_2}{f}$ and $\rmop{supp}b_\iota \subset \pi^{-1}(\rmop{supp}\varphi_{\iota*}\kappa_\iota)$. Hence we can apply Theorem \ref{theo_psido_local_global} and obtain a symbol $b\in S^{m_1+m_2}_f (T^*M)$ such that 
    \[
        \chi_\iota\varphi_\iota^* \Op^1_\hbar (b_\iota)\varphi_{\iota*}\chi_\iota
        =\Op^1_{M, \hbar} (b)+O_{L^2\to L^2}(\hbar^\infty). 
    \]
    $b$ has an asymptotic expansion \eqref{eq_asymptotic_b_III} and the identity \eqref{eq_principal_b} is also valid. 
\end{proof}

\subsection{Proof of main theorem}\label{subs_proof_main}

Before the proof of the main theorem (Theorem \ref{theo_parametrix}), we estimate the symbol $(z-\sigma (P_\hbar))^{-1}$. The ellipticity condition for $P_\hbar$ plays an essential role in the estimate. 

\begin{prop}\label{prop_symb_parametrix}
    Let $P_\hbar \in \mathrm{Diff}^m_{f, \hbar} (M; \Omega^{1/2})$ be an elliptic differential operator. Then, for any complex number $z$ with $\mathrm{dist}(z, \sigma (P_\hbar)(T^*M))>0$, the symbol $(z-\sigma (P_\hbar))^{-1}$ belongs to $S^{-m}_f (T^*M)$. 

    Moreover, the seminorms of $(z-\sigma (P_\hbar))^{-1}$ are estimated as 
    \[
        |(z-\sigma (P_\hbar))^{-1}|_{S^m_f (T^*M), N}
        \leq C\Delta_N (z) 
    \]
    where 
    \begin{align*}
        \Delta_N (z)
        :=&\sum_{l=0}^N \sup_{\substack{ \iota \in I_K \\ (x, \xi)\in \Gamma_\iota}}\frac{\jbracket{\xi}^{m(l+1)}}{|z-\tilde\varphi_{\iota*}\sigma (P_\hbar)(x, \xi)|^{l+1}} \\
        &+\sum_{l=0}^N \sup_{\substack{ \iota\in I_\infty \\ (q, p)\in \Gamma_\iota}}\frac{\jbracket{\rho\oplus f(r)^{-1}\eta}^{m(l+1)}}{|z-\tilde\varphi_{\iota*}\sigma (P_\hbar)(q, p)|^{l+1}}. 
    \end{align*}
\end{prop}

\begin{proof}
    We estimate $(z-\sigma (P_\hbar))^{-1}$ in polar coordinates. We introduce a shorthand notation 
    \[
        \partial_f^\mathfrak{a}=(f(r)^{-1}\partial_\theta)^{\alpha^\prime}(f(r)\partial_\eta)^{\beta^\prime}\partial_r^{\alpha_0} \partial_\rho^{\beta_0} 
    \]
    for $\mathfrak{a}=(\alpha, \beta)=(\alpha_0, \alpha^\prime, \beta_0, \beta^\prime)\in \nni^{2n}$. 
    The derivative of $(z-\sigma(P_\hbar))^{-1}$ for $|\mathfrak{a}|\geq 1$ is the form 
    \[
    \partial_f^\mathfrak{a} (z-\sigma(P_\hbar))^{-1}=\sum_{l=1}^{|\mathfrak{a}|}(z-\sigma(P_\hbar))^{-l-1}c_{l, \mathfrak{a}}, 
    \]
    where $c_{l, \mathfrak{a}}$ is 
    \begin{align*}
    c_{l,\mathfrak{a}}(q, p)
    =\sum_{\substack{\mathfrak{a}=\mathfrak{a}_1+\cdots+\mathfrak{a}_l \\ |\mathfrak{a}_1|, \ldots, |\mathfrak{a}_l|\geq 1}} \partial_f^{\mathfrak{a}_1}\sigma(P_\hbar) \cdots\partial_f^{\mathfrak{a}_l}\sigma(P_\hbar)
    \in S^{ml-|\beta|}_f(T^*M). 
    \end{align*}
    Thus $\partial_f^\mathfrak{a} (z-\sigma(P_\hbar))^{-1}$ is estimated as 
    \[
        |\partial_f^\mathfrak{a} (z-\sigma(P_\hbar)(q, p))^{-1}|
        \leq C\sum_{l=0}^{|\mathfrak{a}|}\frac{\jbracket{\rho \oplus f(r)^{-1}\eta}^{ml-|\beta|}}{|z-\sigma (P_\hbar)(q, p)|^{l+1}}.  
    \]
    The estimate in the case of $\iota\in I_K$ is proved similarly. 
\end{proof}

\begin{proof}[Proof of Theorem \ref{theo_parametrix}]
We define $b_0(z)(x, \xi):=(z-\sigma (P_\hbar)(x, \xi))^{-1}$. We have $(z-\sigma (P_\hbar)(x, \xi))^{-1}\in S^{-m}_f(T^*M)$ by Proposition \ref{prop_symb_parametrix}. By Theorem \ref{theo_diff_psido_composition}, there exists a symbol $e_1(z)(x, \xi)\in S^{-1}_f(T^*M)$ such that the equality 
\begin{equation}
    \label{eq_parametrix_1}
    (z-P_\hbar)\Op^1_{M, \hbar}(b_0(z))=1+\Op^1_{M, \hbar}(e_1(z))+O_{L^2\to L^2}(\hbar^\infty). 
\end{equation}
holds. 

Next we define $b_1(z):=-e_1(z)(z-\sigma (P_\hbar))^{-1}\in S^{-m-1}_f(T^*M)$. Then by Theorem \ref{theo_diff_psido_composition}, there exists a symbol $e_2 (z)\in S^{-2}_f (T^*M)$ such that the equality 
\begin{equation}
    \label{eq_parametrix_2}
    (z-P_\hbar)\Op^1_{M, \hbar}(b_1(z))=-\Op^1_{M, \hbar}(e_1(z))+\Op^1_{M, \hbar}(e_2(z))+O_{L^2\to L^2}(\hbar^\infty)
\end{equation}
holds. Summing up \eqref{eq_parametrix_1} and \eqref{eq_parametrix_2}, we obtain 
\[
    (z-P_\hbar)\Op^t_\hbar (b_0(z)+\hbar b_1(z))=1+\hbar^2\Op^1_{M, \hbar}(e_2(z))+O_{L^2\to L^2}(\hbar^\infty). 
\]
We repeat this procedure and obtain symbols $b_N(z)\in S^{m-N}_f(T^*M)$ and $e_{N+1}(z)\in S^{-N-1}_f(T^*M)$ such that the equality 
\begin{equation}
    \label{eq_parametrix_n}
    (z-P_\hbar)\Op^1_{M, \hbar}\left(\sum_{j=0}^N \hbar^j b_j(z)\right)=1+\hbar^{N+1}\Op^1_{M, \hbar}(e_{N+1}(z))+O_{L^2\to L^2}(\hbar^\infty)
\end{equation}
holds. Since $e_{N+1}(z)\in S^{-N-1}_f (T^*M)\subset S^0_f (T^*M)$, we can apply Theorem \ref{theo_L2_bdd_manifold} and obtain the boundedness of $\Op^1_{M, \hbar}(e_{N+1}(z))$ on $L^2(M; \Omega^{1/2})$. 

Now we define a symbol $b(z)$ as an asymptotic expansion 
\[
    b(z)\sim \sum_{j=0}^\infty \hbar^j b_j(z)
\]
by the Borel theorem. Then 
\[
    (z-P_\hbar)\Op^1_{M, \hbar}(b(z))=1+\hbar^{N+1}\Op^1_{M, \hbar} S^{-N-1}_f(T^*M)
\]
for all $N\in \nni$. 
\end{proof}

\subsection{Essential self-adjointness of elliptic differential operators}\label{subs_ess_selfadj}

As an application, we investigate fundamental properties of elliptic differential operators on manifolds with ends. In this section we focus on the essential self-adjointness of symmetric and elliptic differential operators. 

\begin{theo}\label{theo_self_adjoint}
Suppose that $ P_\hbar \in \mathrm{Diff}^m_{f, \hbar} (M; \Omega^{1/2})$ is elliptic and symmetric in $L^2(M; \Omega^{1/2})$. Then $ P_\hbar $ is essentially self-adjoint in $L^2(M; \Omega^{1/2})$ for sufficiently small $\hbar>0$. 
\end{theo}

\begin{proof}
We prove that $(\pm i- P_\hbar )C_c^\infty( M; \Omega^{1/2} )$ is dense in $L^2(M; \Omega^{1/2})$ for sufficiently small $\hbar >0$. Construct parametrices $\Op^1_{M, \hbar}(b(\pm i))$ as in Theorem \ref{theo_parametrix}. Then 
\[
(\pm i- P_\hbar )\Op^1_{M, \hbar}(b(\pm i)) u=u+ R_{\pm, \hbar}u
\]
for all $u\in C_c^\infty( M; \Omega^{1/2} )$. Here $\|R_{\pm, \hbar}\|_{L^2\to L^2}=O(\hbar^\infty)$. We take $\hbar$ so small that 
\[
\| R_{\pm, \hbar}\|_{L^2\to L^2}<1. 
\]
Then $1+ R_{\pm, \hbar}$ is invertible in $L^2(M; \Omega^{1/2})$. 

Now we take an arbitrary $u\in L^2(M; \Omega^{1/2})$ and $\varepsilon>0$. Since $1+ R_{\pm, \hbar}$ is invertible and $C_c^\infty( M; \Omega^{1/2} )$ is dense in $L^2(M; \Omega^{1/2})$, we can take $v\in C_c^\infty( M; \Omega^{1/2} )$ such that 
\begin{equation}\label{eq_ess_selfadj_approx_1}
\|u-(1+ R_{\pm, \hbar})v\|_{L^2}< \frac{\varepsilon}{2}. 
\end{equation}

Next we take $\chi\in C_c^\infty(\mathbb{R})$ with $\chi(r)=1$ if $|r|<2$. We take a smooth function $\tilde r: M\to [0, \infty)$ such that  
\[
    \tilde r(x):=
    \begin{cases}
        (\rmop{pr}\nolimits_1 \circ \Psi) (x) & \text{if } x\in \Psi^{-1}([2, \infty)\times S), \\
        0 & \text{if } x\in M\setminus \Psi^{-1}([1, \infty)\times S). 
    \end{cases}
\]
Here $\rmop{pr}\nolimits_1 (r, \theta):=r$. 
For each $\iota\in I_\infty$ and $\delta>0$, we define 
\[
 Q_{\pm, \delta, \hbar}:=\chi (\delta \tilde r(x))\Op^1_{M, \hbar}(b(\pm i)). 
\]
We note that $[P_\hbar, \chi(\delta \tilde r)]=\delta \hbar P^\prime_\iota(\delta)$ for some semiclassical differential operator $P^\prime_\hbar (\delta)\in \mathrm{Diff}^{m-1}_{f, \hbar}(M; \Omega^{1/2})$. Then 
\begin{equation}\label{eq_ess_selfadj_approx_2}
\| [P_\hbar, \chi(\delta \tilde r(x))]\Op^1_{M, \hbar}(b(\pm i))v\|_{L^2}=O(\delta) \quad (\delta\to 0) 
\end{equation}
by  Theorem \ref{theo_L2_bdd_manifold} and Theorem \ref{theo_diff_psido_composition}. We combine \eqref{eq_ess_selfadj_approx_1} and \eqref{eq_ess_selfadj_approx_2} and obtain 
\begin{align*}
    &\limsup_{\delta\to +0}\| (\pm i-P_\hbar)Q_{\pm, \delta, \hbar}v-u\|_{L^2} \\
    &\leq \limsup_{\delta\to +0} (\| \chi (\delta \tilde r)(\pm i-P_\hbar)Q_{\pm, \delta, \hbar}v-\chi (\delta \tilde r(x))u\|_{L^2}+\|(1-\chi (\delta \tilde r(x)))u\|_{L^2} \\
    &\quad +\| [P_\hbar, \chi (\delta \tilde r(x))]Q_{\pm, \delta, \hbar}v\|_{L^2}) \\
    &=\| (1+R_{\pm, \hbar})v-u\|_{L^2}\leq \frac{\varepsilon}{2}. 
\end{align*}
This implies the existence of $\delta>0$ such that 
\[
    \| (\pm i-P_\hbar)Q_{\pm, \delta, \hbar}v-u\|_{L^2}<\varepsilon. 
\]
Thus, if we take $w_\pm:=Q_{\pm, \delta, \hbar}v\in C_c^\infty (M; \Omega^{1/2})$ for sufficiently small $\delta>0$, then $\|(\pm i-P_\hbar)w_\pm -u\|_{L^2}<\varepsilon$. 
\end{proof}

As a corollary, we can prove the essential self-adjointness of the Laplacian. We equip $M$ with a Riemannian metric satisfying Assumption \ref{assu_metric_M}. 
In this case we can apply Theorem \ref{theo_self_adjoint} to the Laplacian $-\hbar^2 \triangle_g$. 

\begin{coro}\label{coro_ess_selfadj_laplacian}
    Suppose the Riemannian metric $g$ satisfies Assumption \ref{assu_metric_M}. Then the Laplacian $-\hbar^2\triangle_g$ associated with $g$ is an elliptic differential operator in $\mathrm{Diff}^2_{f, \hbar}( M; \Omega^{1/2})$ and essentially self-adjoint on $L^2(M; \Omega^{1/2})$. 
\end{coro}

\begin{rema*}
    The function $f(r)$ in the symbol class $S^m_f (T^*M)$ and Assumption \ref{assu_metric_M} does not need to be same. For example, we consider a manifold $M$ with hyperbolic ends $g=\diff r^2+e^{2r}h(\theta, \diff \theta)$. Then the associated semiclassical Laplacian $-\hbar^2\triangle_g$ does not only belong to $\mathrm{Diff}^2_{e^r, \hbar}(M; \Omega^{1/2})$, but also belongs to $\mathrm{Diff}^2_{1, \hbar}(M; \Omega^{1/2})$. Thus we can prove Corollary \ref{coro_ess_selfadj_laplacian} by employing the symbol class $S^m_1(T^*M)$, which is easier to employ than the symbol class $S^m_{e^r} (T^*M)$. 

    The symbol class $S^m_{e^r} (T^*M)$ is employed in an analysis of differential operators in $\mathrm{Diff}^m_{e^r, \hbar}(M; \Omega^{1/2})$ such as the Laplacian $-\hbar^2\triangle_g$ associated with the metric $g=\diff r^2+e^{2r}h(r, \theta, \diff \theta)$ satisfying 
    \[
        (e^{-r}\partial_\theta)^{\alpha^\prime}\partial_r^{\alpha_0}h_{jk}(r, \theta)\in L^\infty
    \]
    for any multiindices $\alpha=(\alpha_0, \alpha^\prime)\in \nni \times \nni^{n-1}$. 
\end{rema*}

\section*{Acknowledgements}
The author thanks Professor Kenichi Ito and Professor Shu Nakamura for valuable discussion and useful advice and comments. He especially thanks Professor Shu Nakamura for proposing problems on pseudodifferential calculus in polar coordinates.

\bibliography{mybib_2201}
\bibliographystyle{plain}

\end{document}